\documentclass[11pt]{amsart}

\usepackage{xcolor}
\usepackage{amsmath, amsthm, amssymb, amsfonts, mathtools, enumerate}
\usepackage{hyperref}
\usepackage{verbatim} 
\usepackage{graphicx,epsfig,color}
\usepackage{exscale}

\theoremstyle{plain}

\DeclareMathOperator{\Geo}{Geo}
\DeclareMathOperator{\TGeo}{TGeo}
\newtheorem{theorem}{Theorem}[section]
\newtheorem*{th:re}{Theorem \ref{th:re}}
\newtheorem*{1}{Theorem \ref{th:1}}
\newtheorem*{2}{Theorem \ref{th:2}}
\newtheorem*{3}{Theorem \ref{th:3}}
\newtheorem*{4}{Theorem \ref{th:4}}
\newtheorem*{5}{Theorem \ref{th:5}}

\newtheorem{lemma}[theorem]{Lemma}

\newtheorem{corollary}[theorem]{Corollary}
\newtheorem{conjecture}[theorem]{Conjecture}
\newtheorem{proposition}[theorem]{Proposition}

\theoremstyle{remark}

\theoremstyle{definition}

\newtheorem{definition}[theorem]{Definition}\newtheorem{notation}[theorem]{Notation}

\newtheorem{remark}[theorem]{Remark}

\newtheorem{example}[theorem]{Example}
\newcommand{\ninfty}{\sinfty}

\newtheorem*{jjj}{Remark}

\DeclareSymbolFont{AMSb}{U}{msb}{m}{n}
\DeclareMathSymbol{\N}{\mathalpha}{AMSb}{"4E}
\DeclareMathSymbol{\R}{\mathalpha}{AMSb}{"52}
\DeclareMathSymbol{\Z}{\mathalpha}{AMSb}{"5A}
\DeclareMathSymbol{\D}{\mathalpha}{AMSb}{"44}
\DeclareMathSymbol{\s}{\mathalpha}{AMSb}{"53}

\newcommand{\el}{{l}}

\renewcommand{\Im}{\mbox{Im}}

\newcommand{\sF}{{\scriptscriptstyle{F}}}

\newcommand{\sX}{{\scriptscriptstyle{X}}}
\newcommand{\sXi}{{\scriptscriptstyle{X_i}}}

\newcommand{\sN}{{\scriptscriptstyle N}}

\DeclareMathOperator{\vol}{vol}

\DeclareMathOperator{\spt}{spt}
\DeclareMathOperator{\supp}{supp}

\DeclareMathOperator{\de}{d}

\DeclareMathOperator{\mmm}{m}
\newcommand{\m}{{\mmm}}
\DeclareMathOperator{\ric}{ric}
\DeclareMathOperator{\diam}{diam}

\DeclareMathOperator{\Ent}{Ent}

\newcommand{\sinfty}{{\scriptscriptstyle \infty}}

\newcommand{\gcl}{{-I\times_f F}}

\newcommand{\smwp}{{\scriptscriptstyle{I\times_f M}}}
\newcommand{\smgc}{{\scriptscriptstyle I\times_f F}}
\newcommand{\slgc}{{\scriptscriptstyle -I\times_f F}}

\newcommand{\slwp}{\scriptscriptstyle{-I\times_f M}}

\newcommand{\RCD}{\mathsf{RCD}}
\newcommand{\CD}{\mathsf{CD}}
\newcommand{\TCD}{\mathsf{TCD}}
\newcommand{\MCP}{\mathsf{MCP}}
\newcommand{\TMCP}{\mathsf{TMCP}}

\DeclareMathOperator{\Ric}{ric}
\newcommand{\ol}[1]{{\overline #1}}
\title[Curvature for generalized cones and convergence]{Convergence of Lorentzian  spaces and curvature bounds for generalized cones}
\author{Christian Ketterer}
\thanks{Department of Mathematics \& Statistics, Logic House,
South Campus,
Maynooth University,
Ireland.\ \ {{\it Email address:} {\tt christian.ketterer@mu.ie}}}

\thanks{{\it 2020 Mathmatics Subject Classification.} Primary: 28A75, 51K10, 53C23, 53C50, 53B30, 53C80; \ \ Keywords: Optimal transport, Lorentzian geodesic spaces, generalized cones, curvature bounds}
\address{Department of Mathematics \& Statistics, Logic House,
South Campus,
Maynooth University,
Ireland}\email{christian.ketterer@mu.ie}
\begin{document}
\begin{abstract} The goal of this article is twofold.  We introduce  a notion of convergence for Lorentzian pre-length spaces, $\ell$-convergence, that extends previous convergence notions in this context. We show that timelike curvature and timelike curvature-dimension bounds are stable under (measured) $\ell$-convergence. 
Then, we show that $\ell$-convergence is well adapted for  generalized Lorentzian cones: a sequence of generalized cones $-I_i\times_{f_i}X_i$  converges in $\ell$ sense if the base $I_i$ and the fiber $X_i$ converge in GH sense and the functions $f_i$ converge uniformly. We use this to show sharp timelike curvature and timelike curvature-dimension bounds for such cones. Finally, we obtain a pre-compactness theorem for $\ell$-convergence in the class of smooth generalized  cones that have a uniform lower bound on the full Ricci (or Riemann) curvature tensor. 
\end{abstract}
\maketitle
\tableofcontents
\section{Introducton}
Warped product spaces play a fundamental role  in smooth and nonsmooth Riemannian and Lorentzian geometry.  They are both, model spaces for rigidity and almost rigidity theorems \cite{almostrigidity, coldingshape, beran2025, ketterer3}, and serve as a rich source of  examples and counter-examples to test geometric, analytic, stochastic, and topological properties of spaces statisfying  geometric constraints \cite{cheegercoldingI, coldingnaberII, hnw}. Such constraints  often come in the form of a lower or upper curvature bound and  it is known that curvature bounds are particularly well-behaved under warped product constructions \cite{oneillsemi, albi, albi0, albi3, ketterer2, kettererwp, kettererwp_survey, sou_cones}. 

Here we study  nonsmooth Lorentzian generalized cones. These are warped products  $-I\times_f X$  between  a  base inverval $I$,  a continuous warping function $f: I \rightarrow [0, \infty)$ and  a  fiber metric space $X$.
The minus sign  indicates the Lorentzian character of the space. For a smooth fiber space generalized cones  include the  Robertson-Walker and Friedman spacetimes that serve as cosmological models \cite{oneillsemi, albi_lorentz}.
Nonsmooth generalized cones were first defined and studied in detail  by Alexander, Graf, Kunzinger and S\"amann in \cite{agks}. They prove important  properties such as global hyperbolicity and show that generalized cones are Lorentzian pre-length spaces in the sense of \cite{kusa} (see also ) under quite general assumptions.  Moreover, they show that generalized cones have synthetic  timelike  sectional curvature bounded from below by $-K$ if  
$f$ is  smooth and $\mathcal FK$-concave,  i.e. it satisfies $f''+Kf\leq 0$, and 
$X$ is an Alexandrov space with curvature bounded from below by  $ -\inf (Kf^2 + |f'|^2)=k$. 
Calisti, S\"amann and the author   proved a  theorem for synthetic timelike Ricci curvature lower bounds \cite{cks}: If $X$ is a non-branching metric measure space that satisfies the curvature-dimension condition $\CD((N-1) k,N)$ then the generalized $N$-cone $-I\times^N_f X$ satisfies the measure contraction property $\TMCP(NK,N+1)$. 
The parameter $N$ in $N$-cone refers to reference measure $f^N \de t \otimes \de \m_\sX$. 

Synthetic timelike sectional or Ricci curvature bounds such as  timelike sectional curvature bounded below $TCBB(K)$ \cite{kusa, bkar, bkr}, the timelike entropic curvature-dimension condition $\CD^e_p(K,N)$ and the timelike measure contraction property for Lorentzian prelength space \cite{mccann_lorentz, mondinosuhr, camolorentz} (see also \cite{braun_renyi})  have  attracted a lot of interest in recent years.  These conditions in combination with the theory of Lorentzian length spaces provide 
  frameworks for  nonsmooth extensions of general relativity. Through Einstein's field equations the condition $\CD_p^e(0,N)$ for a smooth spacetime  is equivalent to the strong energy condition that  is crucial in numerous fundamental results about the causal behaviour  of spacetimes  such as the Penrose-Hawking singularity theorems and the Lorentzian splitting theorem. 
Curvature bounds for Lorentzian pre-length spaces also mirror the  theory of synthetic lower sectional and Ricci curvature bounds for metric and metric measure spaces \cite{bbi, bgp, stugeo1, stugeo2, lottvillani}. 

While the literature on Lorentzian length spaces satisfying such curvature conditions has grown substantially in recent years, examples beyond those arising from smooth manifolds remain scarce.
The theorems below  enlarge the current collection of examples by introducing new classes of genuinely nonsmooth Lorentzian spaces satisfying timelike curvature or curvature-dimension bounds.

\smallskip
\paragraph{\bf A. ${\boldsymbol \ell}$-convergence} 

An immediate question for the class of nonsmooth Lorentzian spaces is the identification of a natural analogue of the Gromov–Hausdorff topology for metric spaces that allows for a robust comparison of nonsmooth spacetimes, independent of smoothness assumptions and compatible with curvature bounds. Several proposals have appeared in recent years \cite{camolorentz, ms_gh, om, bs, ms,bms,burtscher_allen, ss, sormani_vega, csp}.  However, available results can  depend on  assumptions that must be checked case by 
case, and the stability of curvature bounds is often not the main focus. 

We will introduce a new  notion of convergence for the class of Lorentzian pre-length spaces (Section \ref{sec:ell}). The main idea roughly is to say that a sequence of signed time separation functions $\ell_i$ associated to Lorentzian pre-length spaces $Y_i$ converges if the $\ell_i$s converge {\it uniformly} as functions on the sequence of product spaces $Y_i \times Y_i$ (Definition \ref{def:uniform_convergence}). For this we require that distance functions $\de_i$  which induce the topology on $Y_i$ can be chosen such  that the sequence $(Y_i, \de_i)$ converges in a GH-type sense (see Definition \ref{def:coveredGH}).  
We call this {\bf ${\boldsymbol\ell}$-convergence} (see Definition \ref{def:ellconvergence}), and measured $\ell$-convergence if there are also weakly converging measures involved.  For the precise definition we assume the existence of  coverings and a condition we call uniformly non-totally imprisoning (Definition \ref{def:uni}). In Section \ref{sec:stability}
we show that timelike curvature and timelike curvature-dimension bounds are stable under (measured) $\ell$-convergence (Theorem \ref{th:stabilityTCBB}, Theorem \ref{th:stabilityTCD}, Theorem \ref{th:stabilityTMCP}).  The notion of $\ell$-convergence includes the convergence type that was used to show stability of the curvature-dimension condition in \cite{camolorentz} (see also \cite{braun_renyi}). Moreover, it  covers uniform convergence of a sequence of  Lorentzian spacetimes  w.r.t. a fixed Riemannian metric on the underlying manifold (Theorem \ref{th:uniform}). Most importantly, we show that a sequence of generalized cones converges in $\ell$-sense if the sequence of bases and the sequence of fibers converge in GH sense and the sequence of warping functions $f$ converges uniformly (Theorem \ref{th:convergencecones}). In fact, $\ell$-convergence is directly motivated by the desire to define a convergence for Lorentzian generalized cones that has this property. 
\smallskip
\paragraph{\bf B. Cuvature of generalized cones}
The second goal of this paper is to employ $\ell$-convergence to establish timelike curvature and curvature-dimension bounds for generalized Lorentzian cones (Section \ref{sec:curcon}).

First, we  extend the previously mentioned theorem by Alexander-Graf-Kunzinger-S\"amann 
 to the case of non-smooth $\mathcal FK$-concave functions.  

\begin{theorem}\label{th:1}
Let $I$ be an open interval,  let $f:  I\rightarrow (0, \infty)$ be continuous, and 
assume $f$ is $\mathcal FK$-concave.  Set $ - \inf_I \{Kf^2 + |\partial f|^2\}=K_f$ and
assume $X$ is an Alexandrov space with curvature bounded from below by $K_f$. Then $-I\times_f X$ has timelike curvature bounded from below $-K$. 
\end{theorem}
\noindent
Here, $|\partial f|$ denotes the local slope of the function $f$.  A function $f: I\rightarrow (0,\infty)$ is said to be $\mathcal FK$-concave ($\mathcal FK$-convex) if  $f''+Kf\leq 0$ ($f''+ Kf \geq 0$) holds in the distributional sense on $I$ \cite{albi}.

The theorem is the optimal Lorentzian version of  the corresponding theorem for Alexandrov spaces with curvature bounded from below by Alexander and Bishop \cite{albi}.  The conditions  are sharp in the sense that if $-I\times_fX$ satisfies the timelike curvature bound and if $f$ is $\mathcal FK$-affine ($\mathcal FK$-concave and $\mathcal FK$-convex) then $X$ has curvature bounded from below by $K_f$ \cite{agks}. However, unlike in the metric case the assumptions are not implied by a timelike curvature bound for the cone in general (Remark \ref{rem:contra}). Convergence of generalized cones and null distances was  studied in \cite{kust_null}. 

If the fiber space $X$ is a Ricci limit space, we show the following theorem for the timelike curvature-dimension condition. 
\begin{theorem}\label{th:2}
Let $f:  I\rightarrow (0, \infty)$ be as before,  and $p\in (0,1)$. 
Assume $X$ is the Gromov-Hausdorff limit of Riemannian manifolds $M_i^n$ with $\ric^{M_i}\geq (n-1) K_f$. Then $-I\times_f X$ satisfies the timelike entropic curvature-dimension conditin $\TCD_p^e(nK, n+1)$. 
\end{theorem}
\noindent
This theorem is the Lorentzian version of the corresponding theorem in the context of spaces with synthetic lower Ricci curvature bounds \cite{kettererwp, ketterer2}. Here, the fiber space is a Ricci limit space. It is conjectured that a statement where the fiber space is a $\CD$ metric measure spaces will hold \cite{kettererwp_survey}.
A partial result in this direction for $f$ smooth was obtained in \cite{cks} and  has the following generalization for nonsmooth $f$.
\begin{theorem}\label{th:3}
Let $f:  I\rightarrow (0, \infty)$ be as before. 
Assume $X$ is a nonbranching $\CD((n-1)K_f, n)$ space. Then $-I\times_f X$ satisfies the timelike entropic measure contraction property $\MCP^e(nK, n+1)$. 
\end{theorem}
\smallskip
\paragraph{\bf Pre-compactness}
In the category of smooth generalized cones, i.e. when $f$ is a smooth function and $X$ is a smooth, complete Riemannian manifold, such that there is lower bound on the {\it full Ricci tensor} of the spacetime, we show a pre-compactness theorem. 
\begin{theorem}\label{th:4}Let $\mathcal M$ be the class of smooth, complete Riemannian manifolds. We define
$$\mathcal Y_{K,N}^{\ninfty}= \{ Y= -I \times_f X: f\in C^2(I),   X\in \mathcal M,\ric^Y\geq K, \dim Y= N \}.$$
The class $\mathcal Y^{\ninfty}_{K,N}$ is precompact w.r.t. measured $\ell$-convergence.
Moreover, any limit space $Y\in \overline{\mathcal Y^{\ninfty}_{K,N}}^{\m\ell}$ satisfies the   condition $\TCD^e_p(K,N)$. 
\end{theorem}
Here, a generalized cone is equipped with the reference measure $f^{N-1}\de t\otimes \vol_X$. We call this an $(N-1)$-cone.

 This theorem is of interest since it only requires a curvature assumption for the full Ricci tensor of the  spacetime, but not a curvature bound for the fiber space (see Subsection \ref{subsec:curv_smooth} for the definition of $\ric^M\geq K$ and also $R^M\geq K$; the latter was introduced \cite{albi_lorentz}). A  pre-compactness theorem for more general smooth spacetimes was also obtained in \cite{ms_gh} but with different  curvature assumptions. 

The theorem can be used to define  tangent cones in a point $p\in Y$ for $Y\in \overline{\mathcal Y^{\ninfty}_{K,N}}^{\m\ell}$ (Subsection \ref{subsec:tangentcone}), and it can be used to show an almost splitting theorem in the class of cones with (full) Ricci curvature bounded from below (Subsection \ref{subsec:almost}).

Unlike for timelike lower Ricci curvature bounds a lower bound on the full Ricci tensor of a Lorentzian spacetime is not implied by a lower bound on the  Riemann tensor $R$ (Subsection \ref{subsec:curv_smooth}). But we still  have the following theorem. 

\begin{theorem}\label{th:5} The class
$$\{ Y= -I \times_f X: f\in C^2(I),   X\in \mathcal M, R^Y\geq K, \dim Y\leq N \}$$is precompact w.r.t. $\ell$-convergence.
Moreover, any limit space satisfies $TCBB(-K)$. 
\end{theorem}
The lower bound on the full curvature  seems to be important for the precompactness results (Remark \ref{rem:thelast}). We formulate the following conjecture.
\begin{conjecture} Let $(M, h)$ be an $n$-dimensional Riemannian manifold.
The class
\begin{center}$\left\{(M,h,g): (M, g) \mbox{ is a globally hyperbolic spacetime, } \ric^{(M,g)}\geq K\right\}$\end{center}
is pre-compact w.r.t. $\ell$-convergence. 
\end{conjecture}

\begin{remark}
The limit spaces that appear in the two previous compactness theorems are spaces that have timelike lower curvature bounds or satisfy a timelike curvature-dimension condition, respectively. But, it is clear that they belong to a more restrictive class of spaces that have curvature bounded from below also in spacelike directions in a sense that generalizes $\ric^M\geq K$ (or $R^M\geq K$). 
%
\end{remark}
\subsection{Overview}
In Section \ref{sec:prel} we will recall basic information about metric and metric measure spaces, Gromov-Hausdorff convergence, Lorentzian pre-length and geodesic spaces, optimal transport on Lorentzian pre-lengthspaces, and curvature and curvature-dimension conditions.

In Section \ref{sec:genc} we will give a brief introduction to generalized cones as metric and Lorentzian spaces, in smooth and non-smooth context. 

In Section \ref{sec:ell} we will introduce $\ell$ and measured $\ell$-convergence, and prove some of it properties. 

In Section \ref{sec:stability} we show that timelike curvature and curvature-dimension bounds are stable w.r.t. $\ell$-convergence. 

In Section \ref{sec:concon} we study $\ell$-convergence of generalized cones. 

In Section \ref{sec:curcon} we use the previous findings to prove timelike curvature and curvature-dimension bounds for generalized cones. 

In Section \ref{sec:preco} we show the pre-compactness w.r.t. $\ell$-convergence of generalized cones that admit a lower bound on the full ricci tensor. 
\subsection*{Declaration on the Use of AI} All mathematical ideas, arguments, and results presented in this
manuscript were developed by the author. After completion of the initial draft, AI tools were used
only to improve some of the language. The author takes full responsibility for the content of the manuscript

\subsection*{Acknowledgements} 
This work was started during a research visit of the author at  the  Universit\'e de Haute-Alsace in Mulhouse in 2025 
funded by the  Institut National des Sciences Math\'ematiques of the CNRS.
I  want to thank  my local host Nicolas Juillet for many inspiring dicussions about topics connected to this work.

%
\section{Preliminaries}\label{sec:prel}
\subsection{Metric and metric measure spaces} We recall some basics about metric and metric measure spaces and their convergence \cite{bbi, agsgradient, gmsstability}.

We consider a  metric space  $(X, \de_\sX)$ that is  separable (locally compact). For $x,y\in X$ we also write $\de_\sX(x,y)= |x,y|$. A pointed metric space $(X, \bar x)$ consists of a complete metric space $X$ and a point $\bar x$. 
%

The length of a continuous curve $\gamma:[a,b]\rightarrow X$ w.r.t. $\de_\sX$ is denoted $L^\sX(\gamma)$. 
The metric space $X$ is said to be a geodesic  space if for all $x, y\in X$ there exists a continous curve $\gamma: [a,b]\rightarrow X$ such that $\gamma(a)=x$, $\gamma(b)=y$ and $L^\sX(\gamma)= |x,y|$.  
\subsubsection{Gromov-Hausdorff convergence {\cite{bbi}}}
A sequence of compact metric spaces $(X_i)_{i\in \N}$ converges in Gromov-Hausdorff (GH) sense to a metric space $X$, we write $X_i \overset{\scriptscriptstyle GH}{\longrightarrow} X$ if $\forall \epsilon>0$ there exists a  map $\phi: X_i \rightarrow X$ such that 
\begin{enumerate}
\item $\left| \de_{\sX}(\phi(x), f(y)) - \de_\sXi(x, y)\right|< \epsilon$,
\item $B_\epsilon(\phi(X_i))= X$. 
\end{enumerate}
The map $\phi: X_i \rightarrow X$ is  called an $\epsilon$-GH isometry. One can choose $\phi$ to be Borel measurable. 
Equivalently, there is a compact metric space $Z$ and isometric embeddings $\iota_{i}: X_i \rightarrow Z$, $\iota: X\rightarrow Z$ such that 
$\iota_i\left(X_i\right)$ converges in the Hausdorff sense to  $\iota\left(X\right).$

 A sequence $(X_i, \bar x_i)_{i\in \N}$ of pointed meric spaces  converges in the pointed  Gromov-Hausdorff (pGH) sense to a pointed metric space $(X, \bar x)$, we write $(X_i, \bar x_i)\overset{\scriptscriptstyle pGH}{\longrightarrow} (X, \bar x)$, if 
for every $R>0$ and $\epsilon>0$ there exists $i(R, \epsilon)\in \N$ such that for every $i\geq i(\epsilon, R)$ there is a   map $\phi: B_R(\bar x_i) \rightarrow X$ such that  $\phi(\bar x_i)= \bar x$ and
\begin{enumerate}
\item $\left| d_\sX(\phi(x), \phi(y))- \de_{\sXi}(x, y)\right|< \epsilon $  $\forall x, y\in B_R(\bar x_i)$; 
\item $B_{R-\epsilon}(\bar x)\subset B_\epsilon( \phi(B_R(\bar x_i)))$.
\end{enumerate}

If $X_i$ are geodesic spaces, then $X$ is a geodesic space. 

If $X$ is a geodesic space, then the  conditions in the definition of pGH convergence can be replace with the following conditions:  

For every $R>0$ and $\epsilon>0$ there exists $i(R, \epsilon)$ such that for every $i\geq i(\epsilon, R)$ there is a map $\phi : B_R(\bar x_i) \rightarrow B_R(\bar x)$ such that  $\phi(\bar x_i)= \bar x$ and
\begin{enumerate}
\item $\left| d_\sX(\phi(x), \phi(y))- \de_{\sXi}(x, y)\right|< \epsilon $  $\forall x, y\in B_R(\bar x_i)$; 
\item $B_{R}(\bar x)\subset B_\epsilon( \phi(B_R(\bar x_i)))$.
\end{enumerate}
In particular,  $\overline{B_R(\bar x_i)}$ converges in GH sense to $\overline{B_R(\bar x)}$.

Moreover, we find a complete separable metric space $Z$ and isometric embeddings $\iota_{i}: X_i \rightarrow Z$, $\iota: X\rightarrow Z$ such that \begin{center}
$\iota_i\left(\overline{B_R(\bar x_i)}\right)$ converges in Hausdorff sense to  $\iota\left(\overline{B_R(\bar x)}\right).$
\end{center} We say that  pGH 	of $X_i$ to $X$ convergence is realized in $Z$.
\subsubsection{Measures and weak convergence}
Let $X$ be a separable metric space. Let $C_b(X)$ and $C_{bs}(X)$ be the space of  bounded continuous function on $X$ and the spaces of continuous functions with bounded support, respectively. 
A sequence of Borel measures $(\m_i)_{i\in \N}$  in $X$ converges weakly to a Borel measure $\m$ in $X$ if 
\begin{align}\label{eq:weakconvergence}\lim_{i\rightarrow \infty} \int f \de \m_i= \int f \de \m \ \ \forall f\in C_{bs}(X).
\end{align}

Let $\mathcal P(X)$ be the set of Borel probability measures. A sequence $(\mu_i)_{i\in \N}\subset \mathcal P(X)$ converges weakly (we also say converges narrowly) to $\mu\in \mathcal P(X)$ if \eqref{eq:weakconvergence} holds with $C_b(X)$ in place of $C_{bs}(X)$. 

Relative arrow compactess in $\mathcal P(X)$ is characterized by Prokhorov's Theorem, i.e.  $\mathcal K \subset \mathcal P(X)$ is precompact w.r.t. to narrow convergence if and only if $\mathcal K$ is tight. A subset $\mathcal K$ is said to be tight if for every $\epsilon>0$ there exists a compact set $K\subset X$ such that $\mu(X\backslash K)\leq \epsilon$ for every $\mu\in \mathcal K$. 

Moreover, a subset $\mathcal K\subset \mathcal P(X\times X)$ is tight if and only if the sets of marginal distribuitions are tight in $X$. For more details we refer to \cite{viltot, gmsstability}.

\subsubsection{Convergence of metric measure spaces} 

A metric measure space  is a triple $(X, \de_\sX, \m_\sX)$ that consists of a separable metric space $(X, \de_\sX)$ equipped with a nonnegative, locally finite (Radon) measure $\m_\sX$. 

A sequence of compact metric measure spaces $(X_i, \de_{\sX_i}, \m_{\sX_i})_{i\in \N}$ converges in measured Gromov Hausdorff (mGH) to a metric measure space $(X, \de_\sX, \m_\sX)$ if $X_i$ converges in GH sense to $X$, and for any sequence $(\epsilon_i)$ with $\epsilon_i\downarrow 0$, and for measurable $\epsilon_i$-GH isometries $\phi_i$ we have that $(\phi_i)_\sharp \m_i \rightarrow \m_i$ weakly in $X$.  Equivalently, if $Z$ is a complete metric space such that $X_i, X$ embed distance preservingly into $Z$ and the sequence $X_i$ converges in Hausdorff sense to $X$, then $\m_i$ converges weakly to $\m$ in $Z$ \cite{GROMOV, stugeo1, gmsstability}.

\subsection{Lorentzian pre-length and geodesic spaces}
Lorentzian pre-length spaces were introduced by Kunzinger and S\"amann in \cite{kusa}. Here, we follow the presentation given in \cite{camolorentz} (see also McCann \cite{camolorentzreview} and Minguzzi \cite{minguzzi} for slightly modified setups).
\subsubsection{Lorentzian pre-length spaces}
A \emph{causal set} is a triple $(Y,\ll, \leq)$ where $Y$ is a set, $\leq$ is a pre-order, i.e. a reflexive and transitive relation, 
and $\ll$ is a transitive relation contained in $\leq$. 
We write $x<y$ when $x\leq y$ and $x\neq y$. We say $x$ and $y$ are \emph{timelike} (resp.\ \emph{causally}) \emph{related} if $x\ll y$ (resp.\ $x\leq y$). 
For  $A\subseteq Y$ we define the \emph{chronological} and \emph{causal future} of $A$ as
\begin{align*}
I^+(A):=\{y'\in Y\colon\exists y\in A, y\ll y'\},  \ \ \ J^+(A):=\{y'\in Y\colon\exists y\in A, y\leq y'\}
\end{align*}
respectively. Analogously we define the chronological past $I^-(A)$ and the causal past $J^-(A)$ of $A$. In the case $A=\{y\}$ for $y\in Y$ we will write $I^+(x):= I^+(\{x\})$ and $J^+(x)=J^+(\{x\})$. 
The \emph{chronological} and \emph{causal emeralds}  of $A,B\subseteq Y$ are defined as
\begin{align}
&I(A,B):=I^+(A)\cap I^-(B), \ \ \ \  J(A,B):=J^+(A)\cap J^-(B)\,.
\end{align}
The \emph{chronological} and \emph{causal diamond} for points $x,y\in Y$ are
 $I(x,y):= I(\{x\}, \{y\})$ and $J(x,y):= J(\{x\}, \{y\})$. 

\begin{definition}
A \emph{Lorentzian pre-length space} $(Y, \de, \ll, \leq, \tau)$ is a causal space $(Y, \ll, \leq)$  with a 
separable metric $\de$ and a 
lower semi-continuous function $\tau: Y\times Y \rightarrow [0,\infty]$ called \emph{time separation function} that satisfies 
\begin{itemize}
    \item[(i)] $\tau(x,y)=0$ if $x\nleq y$, 
    \item[(ii)] $\tau(x,y)>0$ if and only if $x\ll y$, 
    \item[(iii)] $\tau(x,z)\geq \tau(x,y)+ \tau(y,z)$ if $x\leq y\leq z$.
\end{itemize}
\end{definition}
\begin{remark}
The set $Y$ is endowed with the metric topology induced by $\de$. The lower semi-continuity of $\tau$ implies that $I^{\pm}(x)$ is open for all $x\in Y$. We also refer to the function $\tau(x,\cdot)$ \textnormal{(}and $\tau(\cdot,y)$\textnormal{)} as
\emph{Lorentzian distance} from a fixed point $x\in Y$ \textnormal{(}to a fixed point $y\in Y$\textnormal{)}.
\end{remark}
For the signed time separation function $\ell: Y\!\times\! Y\rightarrow \{-\infty\} \cup [0, \infty)$ we set 
$$
\ell(x,y)= \begin{cases} -\infty & \mbox{ if } x\nleq y, \\
\tau(x,y) & \mbox{ otherwise}. 
\end{cases}
$$
The subsets of timelike and causal pairs are then given by
$$
Y^2_{\ll} = \ell^{-1}((0, \infty))= \{\ell>0\}, \ \ \ Y^2_{\leq}= \ell^{-1}([0, \infty))= \{\ell\geq 0\}.
$$
Moreover $\tau= \max\{ 0, \ell\}$.

\begin{notation} The signed time separation function $\ell$ encodes the causal structure of $(Y, \de, \leq , \ll, \tau)$, and we will use  $(Y, \de, \ell)$ as a shorter notation for $(Y, \de, \leq , \ll , \tau)$. 
\end{notation}
\subsubsection{Lorentzian geodesic spaces}
If $I\subseteq\R$ is an interval, a curve $\gamma:I\rightarrow Y$ is called future-directed timelike (resp.\ causal) if $\gamma$ is locally Lipschitz continuous (w.r.t.\ $\de$) and if for all $s\leq t\in I$, it holds $\gamma(s)\ll \gamma(t)$ (resp.\ $\gamma(s)\leq \gamma(t)$). We say that $\gamma$ is a null
curve if, in addition to being causal, no two points on $\gamma(I)$ are timelike related.

The $\tau$-length $L_\tau(\gamma)$ of a future directed causal curve $\gamma\colon[a,b]\rightarrow Y$ is defined via the time separation function, in analogy to arclength in  the theory of  metric spaces \cite{kusa}, i.e.,
\begin{equation*}
    L_\tau(\gamma):=\inf \left\{\sum_{i=0}^{N-1} \tau(\gamma(t_i),\gamma(t_{i+1})): N\in\N,\, a=t_0 < t_1 < \ldots < t_N=b\right\}\,.
\end{equation*}

A future-directed causal curve $\gamma:[a,b]\rightarrow Y$ is called \emph{maximal} (or a \emph{maximizer}) if its length realizes the time separation function between $\gamma(a)$ and $\gamma(b)$, i.e., $L_\tau(\gamma)= \tau(\gamma(a), \gamma(b))$. 

If the time separation function is continuous with $\tau(x,x)=0$ for every $x\in Y$
(as it will be, if we assume that $Y$ is a globally hyperbolic geodesic Lorentzian space), 
then any maximal timelike curve  $\gamma$ with finite $\tau$-length has a (continuous, monotonically strictly increasing) reparametrization $\varphi$ by $\tau$-arclength, i.e. $\tau(\gamma\circ \varphi(s), \gamma\circ \varphi(t))= t-s$ for all $s\leq t$.

A curve $\gamma:[0,1]\rightarrow Y$ will be called
a (causal) \emph{geodesic} if it is maximal and continuous when 
parametrized proportional by $\tau$-arc-length, i.e.   the set of (causal) geodesics is
$$\Geo(Y)= \{\gamma\in {C}([0,1], Y): \tau(\gamma(s), \gamma(t))= (t-s)\tau(\gamma(0), \gamma(1)), \forall s< t\}.$$
Its subset of timelike geodesics is defined as 
$$\TGeo(Y)= \{\gamma\in \Geo(Y): \tau(\gamma(0), \gamma(1))>0\}.$$

   A subset $G\subseteq\TGeo(Y)$  is called \emph{forward timelike nonbranching} if for any $\gamma^1,\gamma^2\in G$ the following holds:
   \begin{align*}
       \exists t\in(0,1)\colon\gamma^1_s=\gamma^2_s\quad\forall s\in[0,t]\quad\Longrightarrow\quad\gamma_s^1=\gamma^2_s\quad\forall s\in[0,1].
   \end{align*}
If $\TGeo(Y)$ is forward timelike nonbranching, then we say the Lorentzian pre-length space $(Y,\de, \ell)$ is forward timelike nonbranching. Similarly, one defines \emph{backward timelike nonbranching}, and we  call the Lorentzian pre-length space \emph{timelike nonbranching} if it is forward and backward timelike nonbranching.

\begin{definition}[Lorentzian geodesic space] A Lorentzian pre-length space $(Y,\de, \ll, \leq, \tau)$ is termed \emph{Lorentzian geodesic space} if additionally it is: 
\begin{itemize}
\item \emph{$\de$-compatible:} every $x\in Y$ admits a neighbourhood $U$ and a constant $C$ such that $L_{\de}(\gamma)\leq C$ for every future or past directed causal curve $\gamma$ contained in $U$;
\item  \emph{geodesic:} for all $x,y\in Y$ with $x<y$ there is a {\it maximal} future-directed causal curve $\gamma$ from $x$ to $y$, i.e.,  $\tau(x,y)= L_{\tau}(\gamma)$.
\end{itemize}
In particular, a Lorentzian geodesic space is \emph{intrinsic}, i.e.,
\begin{equation*}
    \tau(x,y) = \sup\{L_\tau(\gamma): \gamma \text{ future directed causal curve from } x \text{ to } y\}\,,
\end{equation*}
and it is a Lorentzian length space if it is additionally locally causally closed, $I^\pm(x)\neq\emptyset$ for all $x\in Y$ and timelike path connected, see  \cite[Def.\ 3.22]{kusa}. 
\end{definition}
{\color{black} 
A Lorentzian geodesic space is in particular a Lorentzian length space, see
\cite[Def.~3.22]{kusa}.

Hence from \cite[Corollary 3.8]{minguzzi_global} we can consider the following version of global
hyperbolicity that is consistent with the previous literature. 
\begin{definition} A Lorentzian geodesic
space $(X, d, \ll, \leq, \tau)$ is called
\begin{itemize}
    \item {\it Causal}: if $\leq$ is also antisymmetric, i.e.\ $\leq$ is an order;
    \item {\it Globally hyperbolic}: if it is causal and for every $x, y \in X$ the causal
    diamond $J^+(x) \cap J^-(y)$ is compact in $X$.
\end{itemize}
\end{definition}
From \cite[Theorm 3.7]{minguzzi_global} this definition of global hyperbolicity for a Lorentzian geodesic (length) space is equivalent
with the one adopted in \cite{camolorentz}, i.e.  in addition to compactness of causal diamonds $Y$ is non-totally imprisoning. 

\begin{definition} A Lorentzian pre-length space  $Y$ is called {\it non-totally imprisoning} if  for every compact set $K\subset Y$ there exists $C>0$ such that the $\de$-arclength of every causal curve contained in $K$ is bounded by $C$.
\end{definition}

Global hyperbolicity
implies that the relation $\leq$ is a closed subset of $X \times X$, i.e. $X$ is causally closed. It was proved in \cite[Thm.~3.28]{kusa}
that for a globally hyperbolic Lorentzian geodesic space $(X, d, \ll, \leq, \tau)$,
the time-separation function $\tau$ is finite and continuous:
 in particular
the previous remark on the existence of constant $\tau$-speed parametrizations
for maximal causal curves applies, thus any two distinct causally related
points are joined by a causal geodesic.}
\begin{definition} A measured Lorentzian pre-length space is a  Lorentzian pre-length space $(Y, \de, \ell)$ equipped with a non-negative Radon measure $\m$ such that $\spt \m= Y$.  
\end{definition}
Because of (1) the following notion of isomorphism between Lorentzian geodesic spaces is stronger than what is usually required. 
\begin{definition}
We say two  Lorentzian geodesic spaces $(Y_1, \de_1, \ell_1)$ and $(Y_2, \de_2, \ell_2)$ are isomorphic if there exists $\Phi: Y_1 \rightarrow Y_2$ that is bijective and  such that the following properties hold: 
\begin{enumerate}
\item $\Phi$ is distance preserving, 
\item $\Phi$ is $\ell$-preserving:  $\ell_1(x,y)= \ell_2(\Phi(x), \Phi(y))$ $\forall x, y \in Y_1$. 
\end{enumerate}
We say two measured Lorentzian geodesic spaces are isomorphic if in addition the following property holds: 
\begin{enumerate}
\item[(3)] $\Phi_\sharp \m_1= \m_2$. 
\end{enumerate} 

Given two isomorphic Lorentzian geodesic spaces (or measurable Lorentzian geodesic spaces) $Y_1 =(Y_1, \de_1, \ell_1)$ and $Y_2=(Y_2, \de_2, \ell_2)$ (or $Y_i=(Y_i, \de_i, \ell_i, \m_i)$, $i=1,2$)  we write $Y_1\simeq Y_2$.
\end{definition}
\subsection{Optimal transport in Lorentzian length spaces}
We review the theory of optimal transport  developped for Lorentzian pre-length spaces by Eckstein and Miller  in \cite{eckmil},  by McCann in smooth context \cite{mccann_lorentz} and by Cavalletti and Mondino in \cite{camolorentz}.

\subsubsection{The $\ell_p$-optimal transport problem}
 Though the following can be considered in the more general context of Lorentzian pre-length spaces, here we assume  a globally hyperbolic Lorentzian geodesic space $X$.

Let $\mu_0, \mu_1\in \mathcal P(X)$. Let $P_1: X\times X \rightarrow X$ and $P_2:X\times  X \rightarrow X$ the projection on the first and second factor respectively of $X\times X$. We denote 
\begin{align*}
\Pi(\mu_0, \mu_1)&:= \left\{ \pi \in \mathcal P(X\times X): (P_0)_\sharp \pi= \mu_0, \  (P_1)_\sharp \pi=\mu_1\right\}\\
\Pi_{\leq}(\mu_0, \mu_1)&:= \left\{ \pi \in \Pi(\mu_0, \mu_1): \pi(\{ \ell\geq 0\})=1 \right\}
\end{align*}
Clearly, if $\pi \in \Pi_{\leq}(\mu_0, \mu_1)$, then $\pi$-a.e.\ one has $\tau(x,y) = \ell(x,y)$. Moreover,
using the convention that $\infty - \infty = -\infty$, if $\pi \in \Pi(\mu_0, \mu_1)$ satisfies
\[
\int_{X \times X} \ell(x,y)\,\pi(d x, d y) > -\infty,
\]
then $\pi \in \Pi_{\leq}(\mu_0, \mu_1)$. The $\ell_p$-Wasserstein distance between $\mu, \nu\in \mathcal P(X)$ is defined as
\begin{equation}\label{def:lorentz_wasserstein}
\ell_p(\mu, \nu)= \sup_{\pi \in \Pi(\mu,\nu)}
\left( \int_{X \times X} \ell(x,y)^p \,\pi(d x d y) \right)^{1/p}.
\tag{2.9}
\end{equation}

If $X$ globally hyperbolic and geodesic (so that $\tau$ is continuous), 
the cost $\ell^p$ is upper semi-continuous
on $X \times X$. Therefore, one can  invoke standard optimal transport techniques
(e.g.\ \cite{vilot, viltot}) to ensure the existence of a solution of the Monge--Kantorovich
problem (2.9) \cite[Section 2.1]{camolorentz}.

\begin{remark}\label{rem:lorentz_wasserstein} Using standard optimal transport methods one can show that in a globally hyperbolic Lorentzian geodesic spaces, the  supremum in \eqref{def:lorentz_wasserstein} is attained and finite, provided $\Pi_{\leq}(\mu_0, \mu_0)\neq \emptyset$ and $\mu_0$ and $\mu_1$ are compactly supported (Proposition 2.5 in \cite{camolorentz}).
\end{remark}
\noindent
\subsubsection{Cyclical monotonicity}
\begin{definition}[$\ell_p$-ciclical monotonicity]
Fix $p \in (0,1]$ and let $(X,\de, \ell)$ be a Lorentzian pre-length space. A subset $\Gamma \subset X_{\le}$ is said to be $\ell_p$-cyclically monotone if, for any $N \in \mathbb{N}$ and any family $(x_1,y_1), \dots, (x_N,y_N)$ of points in $\Gamma$, the following inequality holds:
$$
\sum_{i=1}^N \ell(x_i,y_i)^p \ge \sum_{i=1}^N \ell(x_{i+1},y_i)^p
$$
with the convention $x_{N+1} = x_1$.

A coupling is said to be  $\ell_p$-cyclically monotone if it is concentrated on a  $\ell^p$-cyclically monotone set.
\end{definition}
The following definition appears in \cite{camolorentz}.
\begin{definition}[(Strongly) Timelike $p$-dualisable measures]
Let $(X, d, \ll, \leq, \tau)$ be a Lorentzian pre-length space and let $p \in (0,1]$. A pair $(\mu, \nu) \in \mathcal{P}(X)^2$ is \emph{timelike $p$-dualisable} (by $\pi \in \Pi_{\ll}(\mu, \nu)$) if
\begin{enumerate}
    \item $\ell_p(\mu, \nu) \in (0,\infty)$;
    \item $\pi \in \Pi^{p\text{-opt}}_{\leq}(\mu, \nu)$ and $\pi(X_{\ll}^2) = 1$;
    \item there exist measurable functions $a,b : X \to \mathbb{R}$, with $a \oplus b \in L^1(\mu \otimes \nu)$ such that
   $
        \ell_p \leq a \oplus b \quad \text{on } \operatorname{supp} \mu \times \operatorname{supp} \nu.
   $
\end{enumerate}
The pair  $(\mu, \nu) \in \mathcal{P}(X)^2$ is called \emph{strongly timelike $p$-dualisable} if, in addition, it satisfies:
\begin{enumerate}
    \setcounter{enumi}{3}
    \item there exists a measurable $\ell_p$-cyclically monotone set $\Gamma \subset X_{\ll}^2 \cap (\operatorname{supp} \mu \times \operatorname{supp} \nu)$ such that a coupling $\pi \in \Pi_{\leq}(\mu, \nu)$ is $\ell_p$-optimal if and only if $\pi(\Gamma) = 1$.
\end{enumerate}
\end{definition}
\subsubsection{Geodesics of probability measures in Lorentzian geodesic spaces}
The evaluation map is defined by
\begin{equation}
\label{eq:eval_map}
e_t : C([0,1], X) \to X, \quad \gamma \mapsto e_t(\gamma) := \gamma_t,
\quad \forall t \in [0,1].
\end{equation}

Let $(X, d, \ll, \leq, \tau)$ be a Lorentzian pre-length space and let $p \in (0,1]$. We say that $\eta \in \mathcal{P}(\mathrm{Geo}(X))$ is an $\ell_p$-optimal dynamical plan between $\mu_0 \in \mathcal{P}(X)$ and $\mu_1 \in \mathcal{P}(X)$ if 
$
(e_0)_\# \eta = \mu_0, \  (e_1)_\# \eta = \mu_1
$
and
\begin{equation}
\label{eq:optimal_coupling}
(e_0, e_1)_\# \eta \in \Pi^{p\text{-opt}}_{\leq}\big(\mu_0, \mu_1\big).
\end{equation}
The set of $\ell_p$-optimal dynamical plans from $\mu_0$ to $\mu_1$ is denoted by
\[
\mathrm{OptGeo}_{\ell_p}(\mu_0, \mu_1).
\]

We say that a curve $[0,1] \ni t \mapsto \mu_t \in \mathcal{P}(X)$ is an $\ell_p$-geodesic if there exists an $\ell_p$-optimal dynamical plan $\eta$ between $\mu_0$ and $\mu_1$ such that
\[
\mu_t = (e_t)_\# \eta, \quad \forall t \in [0,1].
\]
Notice that if $\eta \in \mathrm{OptGeo}_{\ell_p}(\mu_0, \mu_1)$, then the  corresponding $\ell_p$-geodesic
$
(\mu_t)_{t\in [0,1]}
$
is continuous  w.r.t.  weak convergence and satisfies
\[
\ell_p(\mu_s, \mu_t) = (t - s)\,\ell_p(\mu_0, \mu_1),
\quad \forall s,t \in [0,1].
\]
\cite[Theorem 2.32]{camolorentz}
\begin{remark} Let $X$ be a globally hyperbolic Lorentian geodesic space. 
If $\mu_0, \mu_1 \in \mathcal{P}_c(X)$ have compact support such that $\Pi_{\leq}^{p-opt}(\mu_0, \mu_1)\neq \emptyset$, then there always exists an $\ell_p$-optimal dynamical plan $\eta \in \mathrm{OptGeo}_{\ell_p}(\mu_0, \mu_1)$ (and thus an $\ell_p$-geodesic) from $\mu_0$ to $\mu_1$, see~\cite[Prop.~2.33]{camolorentz} for the proof and further properties of $\ell_p$-optimal dynamical plans.
\end{remark}
\subsection{Curvature-dimension conditions}
For $\kappa\in \mathbb{R}$ let $\sin_{\kappa}:[0,\infty)\rightarrow \mathbb{R}$ be the solution of 
$
v''+\kappa v=0, \ v(0)=0 \ \ \& \ \ v'(0)=1.
$
We set \begin{center} $\pi_\kappa=\begin{cases}  \frac{ \pi}{\sqrt \kappa}& \mbox{ if } \kappa>0, \\
\infty &\mbox{ if } \kappa\leq 0.\end{cases}$
\end{center}

For $K\in \mathbb{R}$, $N\in (0,\infty)$ and $\theta> 0$ we define the \textit{distortion coefficient} as
\begin{eqnarray}
t\in [0,1]\mapsto \sigma_{K,N}^{(t)}(\theta)=\begin{cases}
                                             \frac{\sin_{K/N}(t\theta)}{\sin_{K/N}(\theta)}\ &\mbox{ if } \theta\in (0,\pi_{K/N}),\\
                                             \infty\ & \ \mbox{otherwise}.
                                             \end{cases}
\end{eqnarray}
One sets $\sigma_{K,N}^{(t)}(0)=t$.
Moreover, for $K\in \mathbb{R}$, $N\in [1,\infty)$ and $\theta\geq 0$ the \textit{modified distortion coefficient} \cite{stugeo2} is defined as
\begin{eqnarray*}
t\in [0,1]\mapsto \tau_{K,N}^{(t)}(\theta)=
                                            t^{\frac{1}{N}}\left[\sigma_{K,N-1}^{(t)}(\theta)\right]^{1-\frac{1}{N}}  
                                     \end{eqnarray*}
                                           and $\tau_{K,N}^{(t)}(\theta)=\theta\cdot\infty  \mbox{ if }K>0\mbox{ and }N=1$.

Let $(X, \de, \m)$ be a metric measure space. For $N\in [1, \infty)$ the \textit{$N$-Renyi entropy} w.r.t. $\m$ is defined by
$$
S_N(\cdot|\m):\mathcal{P}^2_b(X)\rightarrow (-\infty,0],\ \ S_N(\mu|\m)=\begin{cases}-\int_X \rho^{1-\frac{1}{N}}\de\m& \mbox{if $\mu=\rho\m$,  }\smallskip\\
0& \mbox{otherwise}.
\end{cases}
$$
By Jensen's inequality we have $S_N(\mu|\m)\geq - \m(\supp \mu)^{\frac{1}{N}}$ for every $\mu \in \mathcal P_c(X)$.  The $N$-Renyi entropy is lower semi-continuous w.r.t. weak convergence in $\mu$ and in $\m$. 

The Shanon-Boltzmann entropy w.r.t. $\m$ is define by 
\begin{align*}\Ent_{\m}(\mu)= \begin{cases} \int_X \rho \log \rho \de\m & \mbox{if } \mu= \rho \m \mbox{ and } (\rho \log \rho)^+ \in L^1(\m), \\
\infty & \mbox{otherwise}
\end{cases}\end{align*}
and let $D(\Ent_{\m})$ be the domain of finiteness of $\Ent_\m$. Also $\Ent_\m(\mu)$ is  lower semi-continuous w.r.t. weak convergence in $\mu$ and in $\m$. 

We recall the following definition from \cite[Def. 3.2, Prop. 3.3]{camolorentz}. The functional $U_N: \mathcal P_c(M)\rightarrow [0, \infty)$, $N\in (0, \infty)$, is defined by 
$$U_N(\mu)= e^{-\Ent_\m(\mu)/N}.$$
\begin{definition}[{\cite{camolorentz}}] Let $K\in\R$ and $N\in (0,\infty)$. 
 A measured Lorentzian pre-lenght space $Y$ satisfies the {entropic timelike curvature-dimension condition} ${\TCD_p^e(K,N)}$ if for every timelike $p$-dualizable pair $\mu_0,\mu_1 \in D(\Ent_\m)\cap \mathcal P_c(Y)$, there exists an ${\ell_p}$-geo\-desic $(\mu_t)_{t\in [0,1]}$ connecting $\mu_0$ and $\mu_1$  as well as a timelike $p$-dualizing  coupling $\pi\in{\Pi^{p-opt}_\ll(\mu_0,\mu_1)}$ such that $\forall t\in [0,1]$,
\begin{align*}
U_N(\mu_t) \geq \sigma_{K,N}^{(1-t)}\big[\left\| \tau \right\|_{L^2(\pi)}\big]U_N(\mu_0) + \sigma_{K,N}^{(t)}\big[\left\| \tau \right\|_{L^2(\pi)}\big]U_N(\mu_1).
\end{align*}
 If the previous claim holds for every strongly timelike $p$-dualizable  $\mu_0,\mu_1\in D(\Ent_\m)\cap \mathcal P_c(Y)$,  $Y$ satisfies the \emph{weak entropic timelike curvature-dimension condition} ${ w\TCD_p^e(K,N)}$.
\end{definition}
\begin{definition}[{\cite{braun_renyi}}]
Let $p \in (0,1)$, $K \in \mathbb{R}$, and $N \in [1,\infty)$.
We say that $X$ obeys the timelike curvature-dimension condition $\mathrm{TCD}_p(K,N)$ if for every timelike $p$-dualizable pair $\mu_0,\mu_1\in \mathcal{P}^{\mathrm{ac}}_c(X,m)$, there exists an $\ell^p$-geodesic $(\mu_t)_{t \in [0,1]}$ connecting $\mu_0=\rho_0 \m$ to $\mu_1= \rho_1 \m$, as well as a timelike $p$-dualizing coupling $\pi \in \Pi_{\ll}(\mu_0,\mu_1)$ such that for every $t \in [0,1]$ and every $N' \ge N$,
\begin{align*}
S_{N'}(\mu_t)\le - \int\left[\tau^{(1-t)}_{K,N'}\big(\theta\big)\,\rho_0(x_0)^{-\frac 1 {N'}}  + \tau^{(t)}_{K,N'}\big(\theta\big) \rho_1(x_1)^{-\frac 1 {N'}} \right]\de\pi(x_0,x_1).
\end{align*}
where $\theta= \tau(x_0, x_1)$. 
If the previous statement holds only for every strongly timelike $p$-dualizable pair $\mu_0,\mu_1\in \mathcal{P}^{\mathrm{ac}}_c(M^2,m)$, we say that $X$ obeys the weak timelike curvature-dimension condition $\mathrm{wTCD}_p(K,N)$.

If one replaces the distortion coefficients $\tau_{N',K}^{(t)}(\theta)$ with $\sigma_{N', K}^{(t)}(\theta)$ then we say $X$ obeys the reduced (weak) timelike curvature-dimension condition $\TCD^*_p(K,N)$ ($w\TCD^*_p(K,N)$). 
\end{definition}
\begin{remark}
Assuming that the Lorentzian pre-length space $X$ is geodesic and $p$-essentially nonbranching, Braun showed in \cite{braun_renyi}  that the conditions $\TCD^*_p$, $\TCD^e_p$, $w\TCD^*_p$ and $w\TCD^e_p$ are all equivalent.  Being $p$-essentially nonbranching is a measure theoretic condition that is implied in particular if the space $X$ is a nonbranching Lorentzian geodesic space. 
\end{remark}
Weaker variants of curvature-dimension conditions in the context of metric measure spaces are obtained by considering only $\ell_p$-geodesics where $\mu_1$ is the Dirac delta measure in a point $x_1$.  Such conditions go under the name of Measure Contraction Property ($\MCP$) and were developped by Sturm and Ohta \cite{stugeo2, ohtmea}. The following entropic timelike measure contraction property for Lorentzian length spaces was introduced in \cite{camolorentz}. 
\begin{definition}
We fix $p\in (0, 1)$, $K\in \R$ and $N\in (0, \infty)$. A measured Lorentzian pre-length space $(X, \de, \leq, \ll, \ell, \m)$ satisfies the entropic timelike measure contraction property $\TMCP^e(K,N)$ if for any $\mu_0 \in \mathcal P_c(X) \cap D(\Ent_\m)$ and for any $x_1\in X$ such that $\mu_1(I^-(x_1))=1$, there exists an $\ell_p$-geodesic $(\mu_t)_{t\in [0,1]}$ from $\mu_0$ to $\mu_1=\delta_{x_1}$ such that 
$$U_N(\mu_t|\m)\geq \sigma_{K,N}^{(1-t)}\left( \left\| \tau(\cdot, x_1)\right\|_{L^2(\mu_0)}\right) U_N(\mu_0|\m).$$
\end{definition}
\begin{remark} 
\begin{enumerate}
\item The validity of $\TMCP^e$ is independent of the choice of $p\in (0,1)$. 
This follows from the fact that the only (and therefore optimal) coupling between $\mu_0\in \mathcal P_c(X)$ and $\delta_{x_1}$ is $\pi= \mu_0\otimes \delta_{x_1}$. 
\item The weak timelike entropic curvature-dimension condition $w\TCD^e(K,N)$ implies $\TMCP^e(K,N)$ \cite[Proposition 3.12]{camolorentz}.
\item Under the assumption of timelike nonbranching the condition $\TMCP^e$ implies sharp Bonnet-Myers and Bishop-Gromo-type estimates \cite{camolorentz}
\item As for the curvature-dimension conditions, there are other variants of timelike measure contraction properties, such as $\TMCP$ and $\TMCP^*$, introduced in \cite{braun_renyi}.
\end{enumerate}
\end{remark}

\subsection{Timelike sectional curvature bounds}
We recall the synthetic definition of sectional curvature bounded from below for Lorentzian pre-length spaces via $4$-point configurations \cite{bkr}. For Lorentzian geodesic spaces this is equivalent to the definition via timelike triangle comparison proposed by Kunzinger and S\"amann \cite{kusa}. The latter was motivated by previous work of Alexander, Bishop and Harris \cite{albi_lorentz, harris}.

The 2D Lorentzian model spaces of constant curvature $K \in \mathbb{R}$ are 
\begin{center}
$
\mathbb L^2(K) :=
\begin{cases}
\widetilde{\mathbb S}^2_1\!\left(\frac{1}{\sqrt{K}}\right), & K > 0, \\
\mathbb{R}^{2,1}, & K = 0, \\
\widetilde{\mathbb H}^2_1\!\left(\frac{1}{\sqrt{-K}}\right), & K < 0.
\end{cases}
$
\end{center}
Here $\widetilde{\mathbb S}^2_1(r)$ is the  universal cover of the $2$-dimensional Lorentzian pseudosphere of radius $r > 0$, $\mathbb{R}^{2,1}$ is $2$-dimensional Minkowski spacetime, and $\widetilde{\mathbb H}^2_1(r)$ is the universal cover of the $2$-dimensional Lorentzian pseudohyperbolic space.

\begin{definition}[4-point configurations]
Let $(X,\de, \ell)$ be a Lorentzian pre-length space.
\begin{enumerate}
\item[(i)] A timelike  future endpoint-causal $4$-point configuration is a quadruple $(y,x,z_1,z_2)\in X^4$ such that $y \ll x \ll z_1\leq z_2$.
\item[(ii)] A timelike past endpoint-causal $4$-point configuration is a quadruple $(z_2,z_1,x,y)\in X^4$ such that $ z_2 \leq z_1 \ll x \ll y$.
\item[(iii)] Given a 4-point configuration $(y,x,z_1,z_2)$, that is timelike future endpoint-causal,  and  $K \in \mathbb{R}$,  a 4-point comparison configuration in $\mathbb L^2(K)$ is a quadruple $(\bar{y},\bar{x},\bar{z}_1,\bar{z}_2)\in \mathbb L^2(K)^4$ such that
\begin{enumerate}
\item[(a)] $\tau(y,x) = \bar{\tau}(\bar{y},\bar{x})$,
\item[(b)] $\tau(y,z_i) = \bar{\tau}(\bar{y},\bar{z}_i)$, $i=1,2$,
\item[(c)] $\tau(x,z_i) = \bar{\tau}(\bar{x},\bar{z}_i)$, $i=1,2$, and 
\item[(d)] $\bar{z}_1,\bar{z}_2$ lie on opposite sides of the line through $\bar{y},\bar{x}$.
\smallskip
\end{enumerate}
\item[(iv)] Similarly, one defines a 4-point comparison configuration for a timelike past endpoint-causal  4-point configuration.
\end{enumerate}
\end{definition}
\begin{definition}[4-point condition]\label{def:4pointcondition}
Let $(X, \de,\ell)$ be a Lorentzian pre-length space and $K \in \mathbb{R}$. A $\ge K$-comparison neighborhood is an open set $U \subseteq X$ such that
\begin{enumerate}
\item[(i)] the time separation function $\tau$ is continuous on the open set $(U \times U)\cap \tau^{-1}([0,\pi_{-K}))$, and
\item[(ii)] for every 4-point configuration $(y,x,z_1,z_2)$ in $U$, that is  timelike future endpoint-causal,  with $\tau(y,z_2) < \pi_{-K}$ and its comparison configuration $(\bar{y},\bar{x},\bar{z}_1,\bar{z}_2)$ in $\mathbb L^2(K)$ one has
$$
\tau(z_1,z_2) \ge \bar{\tau}(\bar{z}_1,\bar{z}_2).
$$
Moreover, for every timelike past endpoint-causal 4-point configuration $(z_2,z_1,x,y)$ in $U$ with $\tau(z_2,y) < \pi_{-K}$ and its comparison configuration $(\bar{z}_2,\bar{z}_1,\bar{x},\bar{y})$ in $L^2(K)$ one has
\[
\tau(z_2,z_1) \ge \bar{\tau}(\bar{z}_2,\bar{z}_1) .
\]
\end{enumerate}
Finally, we say that $(X,\ell)$ has timelike sectional curvature bounded below  by $K$, {\color{black} i.e. $(X, \ell)$ satisfies $TCBB(K)$,} if $X$ can be covered by $\ge K$-comparison neighborhoods, and we say that $(X,\ell)$ has global timelike sectional curvature bounded below by $K$ if $X$ is a $\ge K$-comparison neighborhood.

\end{definition}
The 4-point condition is equivalent to  other synthetic timelike sectional curvature bounds for large classes of Lorentzian pre-length spaces, including Lorentzian geodesic spaces, and hence to smooth timelike sectional curvature bounds, see \cite[Thm.~5.1]{bkr} and \cite{bkar}.

{\color{black}
In analogy to Alexandrov spaces with curvature bounded below it was shown in \cite{bhnr} that under quite general assumptions timelike sectional curvature bounded from below by $K$ implies the corresponding global condition. More precisely, this holds if the Lorentzian prelength space is connected, globally hyperbolic, regular and admits a time function.}
\subsection{Curvature bounds in smooth context}\label{subsec:curv_smooth}
Let $(M, g)$ be a spacetime that has signature $(-, +, \dots, +)$, and let $-M=(M, -g)$ be the spacetime with signature $(+, -, \dots, -)$. We also write $g=g_M=\langle \cdot, \cdot \rangle$. Let $R^M= R^{-M}=R$ be  the Riemann curvature tensor. 
\subsubsection{Sectional curvature}
Let $v, w\in T_pM$ such that $\{v, w\}= P$ is a non-degenerate $2$-plane in $T_pM$.  $\{v, w\}$ denotes the linear span of $v$ and $w$.  The sectional curvature of $P$ is 
$$K(P)= \frac{ \langle R(w, v)v, w\rangle}{\langle v, v\rangle \langle w, w\rangle - \langle v, w\rangle ^2}.$$
We set $Q(v,w)= \langle v, v\rangle \langle w, w\rangle - \langle v, w\rangle ^2$. We call a  $2$-plane $P$ that is spanned by two orthogonal vectors $v,w$ timelike if $Q(v,w)=\langle v, v\rangle \langle w, w\rangle <0$. If $Q(v,w)>0$, we call $P$ spacelike. 

A spacetime $(M, g)$ has sectional curvature bounded from below $K$ \cite{albi_lorentz}, we write $R\geq K$, if
$$ \langle R(v, w)w, v\rangle \geq K Q(v,w).$$
Equivalently, for every timelike plane $P$, we have $K(P)\leq K$, and for every spacelike plane $P$, we have $K(P)\geq K$. 

\begin{remark}\label{rem:RK}
\begin{enumerate}
\item If we consider $-M$, then  $R= R^M= R^{-M}$, but $K=K^M= - K^{-M}$. 
Consequently, $K^M\leq  K$ $\Leftrightarrow$ $K^{-M}\geq  -K$.

\item In general, sectional curvature bounded from below by $K$ does not imply sectional curvature bounded from below by $K'$ for $K\geq K'$. 
\item If we consider $\lambda M$, i.e. $M$ with $\lambda^2 g_M$, then $R^M\geq K$, implies $R^{\lambda M}\geq \frac 1 {\lambda^2} K$.
\end{enumerate}

It is now standard convention to say that $(M,g)$ has \emph{timelike sectional curvature bounded from below by $K$} if $K(P)\leq K$ for every timelike plane.  From \cite{bkar} we know that this is consistent with Definition \ref{def:4pointcondition} and the  definition proposed in \cite{kusa} for general Lorentzian pre-length spaces. 
\end{remark}
\subsubsection{Ricci curvature}
Let $e_1, \dots, e_n$ be an orthonormal basis of $T_pM$ w.r.t. $g_M$.
The Ricci tensor of $M$ is define by 
\begin{center}$v, w\in T_pM\mapsto \ric^M(v,w)= \sum_{i=1}^{n}\varepsilon_i \langle R(e_i, v)w, e_i\rangle$\end{center}
where $\varepsilon= g_M(e_i, e_i)$, $i=1, \dots, n$ \cite{oneillsemi}
Consequently, for the Ricci tensor of $-M$ we have
$\ric^{-M}(v,w)
=\ric^M(v,w).$
\begin{definition}[{\color{black} Lower bounds for the  Ricci tensor}]
We say that $\ric^M$ is bounded from below by $K$, we write $\ric^M\geq K$, if 
$$\ric^M(v,v)\geq -K g_M(v,v) \ \ \forall v\in TM.$$
It holds $\ric_M\geq K$ if and only if 
\begin{enumerate}
\item $\ric^M(v,v)\geq - K g_M(v,v)= K | v|^2$ for every timelike vector $v$, and 
\item $\ric^M(v,v)\geq - K g_M(v,v)= -K | v|^2$ for every spacelike vector $v$.
\end{enumerate}

We  say $M$ has \emph{timelike Ricci curvature bounded from below by $K$} if only the first condition (1) is required. 
\end{definition}
\begin{remark}
A spacetime $M$ has timelike Ricci curvature bounded from below by $K$ and dimension bounded from above by $N$  if and only if  a timelike curvature-dimension condition holds, such as $\TCD(K,N)$, $\TCD^*(K,N)$ or $\TCD^e(K,N)$ \cite{mccann_lorentz, mondinosuhr} (see also \cite{braun_renyi}). 
\end{remark}

\begin{remark}
Let us assume that $M$ has timelike sectional curvature bounded from below $K$.

Let $v\in T_pM$ be a timelike vector and let $e_1=v, \dots, e_n$ be an orthonormal basis of $T_pM$ w.r.t. $g_M$.  It follows that $P=\{ v, e_i\}$, $i=2, \dots, n$, are nondegenerated, timelike $2$-planes, $\varepsilon_i = g_M(e_i, e_i)=1$, $i=2, \dots, n$,  and 
\begin{align*}\ric^M(v,v)&= \sum_{2=1}^n \varepsilon_i \langle R(e_i, v)v, e_i\rangle = \sum_{i=2}^n \langle R(e_i, v)v, e_i\rangle \\
&\geq \sum_{i=2}^n K Q(v, e_i)= K\sum_{i=2}^n g_M(v,v)= (n-1)K g_M(v,v).\end{align*}
Hence, $M$ has timelike Ricci curvature bounded from by $-(n-1)K$.
\end{remark}
\section{Generalized cones}\label{sec:genc}
Let $I\subset \R$ be an open interval and let $f: I\rightarrow (0,\infty)$ be continuous. Let $F$ be a geodesic metric space.
We set $\overline I = I \cup \partial I$. 
\subsection{Generalized cones as metric spaces}\hspace{0mm}
We recall the theory of metric warped products between metric spaces developped by Alexander and Bishop in \cite{albi, albi0}.

Let  $\gamma:[a,b] \rightarrow I\times F$ be  a curve such that 
 $P_1\circ \gamma=:\alpha$ and $P_2\circ \gamma=:\beta $ are Lipschitz continuous curves. In the following we call such  $\gamma$ admissible. A length structure $L^\smgc$ on the class of admissible curves is given by
$$L^\smgc(\gamma)= \int_a^b \sqrt{(\alpha')^2 + (f\circ \alpha)^2 |\beta'|^2}.$$ 
Here  $\alpha'$ is the derivative of $\alpha$ and $|\beta'|$ is the metric speed of $\beta$ that are both defined $\mathcal L^1$-a.e. in $I$. 

The warped product {\it semi-metric} $\de_{\scriptscriptstyle I\times_f F}$ on $I\times F$ is  the intrinsic distance associated to the length structure $L^\smgc$, i.e. for 2 points $(s,x)$ and $(t,y)$ we define 
$\de_{\smgc}((s,x), (t,y)):= \inf L^\smgc(\gamma)$
where the infimum is  w.r.t. rectifiable curves $\gamma$ that connect the points $(s,x)$ and $(t,y)$ in $I\times F$.  
$\de_{\smgc}$ is symmetric and satisfies the $\triangle$-inequality.

\begin{definition} The warped product metric space $B\times_f F$ between $\overline I = B$, $F$ and $f$ is given by 
$$( B\times F/\sim, \de_{\smwp}) \  \mbox{where } (p,x)\sim (q,y)  \Longleftrightarrow  \de_{\smwp}((p,x), (q,y))=0.$$
The metric space $B\times_f F$ is  intrinsic. 
\end{definition}\begin{theorem}[Alexander-Bishop, {\color{black} \cite[Theorem 3.1]{albi0}}]\label{th:albi0}
Let $\gamma=(\alpha, \beta)$ be a minimizer w.r.t. $L^{\scriptscriptstyle I \times_f F}$ in $B\times_f F$ parametrized proportional to arclength. Assume {\color{black} $f$ is continuous and }$f>0$. Then 
\begin{enumerate}
\item[(a)] $\beta$ is a minimizer in $F$; 
\item[(b)] (Fiber independence) $\alpha$ is independent of $F$, except for the total height, i.e. the length $L^\sF(\beta)$ of $\beta$. More precisely, if $\hat F$ is another strictly intrinsic metric space and $\hat \beta$ is a minimizing geodesic in $\bar F$ with the same length and speed as $\beta$, then $(\alpha, \hat \beta)$ is a minimizer in $B\times_f\hat F$. 
\item[(c)] (Energy equation, version 1) $\beta$ has speed $\frac{c_\gamma}{f^2\circ \alpha}$ for a constant $c_\gamma$;
\item[(d)] (Energy equation, version 2) 
$\alpha$ satisfies $\frac{1}{2} |\alpha'|^2 + \frac{1}{2f^2\circ \alpha}= E$ a.e. where $E$ is the proprotionality constant of the parametrization of $\gamma$. 
\end{enumerate}
\end{theorem}
\subsection{Generalized cones a Lorentzian length spaces}\hspace{0mm}
 In the following  we recall  the theory of Lorentzian generalized cones developped by Alexander, Graf, Kunzinger and S\"amann in \cite{agks}.

Let  $\gamma: [a,b]\rightarrow I\times F$ be again  admissible, i.e. $\alpha$ and $\beta$ are Lipschitz continuous.
We call $\gamma$ 
\begin{align*}
\begin{Bmatrix} \mbox{timelike}\\
\mbox{null} \\
\mbox{causal} 
\end{Bmatrix} \mbox{ if } \  -(\alpha')^2 + (f\circ \alpha)^2 |{\beta}'|^2\ \  \begin{Bmatrix} <0 \\
=0\\
\leq 0
\end{Bmatrix} \mbox{ $\mathcal L^1$-a.e. in }  [a,b].
\end{align*} 
The path $ \gamma$ is called future/past directed if $ \alpha$ is strictly monotonically increasing/decreasing, i.e. $ \alpha'>0$ or $ \alpha'<0$ a.e.  
\begin{remark}
Points $(s,x)$ and $(t,y)$ in $I\times F$ are chronologically related, denoted by $(s,x)\ll (t,y)$, if there exists a future directed timelike curve from $(s,x)$ to $(t,y)$. Moreover, points $(s,x)$ and $(t,y)$ are causally related, denoted with $(s,x)\leq (t,y)$, if there exists a future directed causal curve from $(s,x)$ to $(t,y)$, or $(s,x)= (t,y)$. 
\end{remark}

Let $\gamma:[a,b]\rightarrow I\times F$ be a future directed causal curve.  Then a Lorentzian  length structure is defined via
$$L^\slgc(\gamma)= \int_a^b \sqrt{ (\alpha')^2 - (f\circ \alpha)^2 |\beta'|^2} \de t. $$
\begin{proposition}[{\cite[Proposition 3.15]{agks}}]\label{prop:limsup}
Let $\gamma_i: [a,b] \rightarrow I\times F$, $i\in \N\cup\{\infty\}$, be causal curves such that $\gamma_i\rightarrow \gamma_\sinfty$ pointwise. 
Then
$$\limsup_{i\rightarrow \infty} L^\slgc(\gamma_i)\leq L^\slgc(\gamma).$$
\end{proposition}
For two points $(p,x)$ and $(q,y)$ with $(p,x)\leq (q, y)$  the time separation function is defined via
$$
\tau_{\slgc}((p,x), (q,y)):= \sup L^\slgc(\gamma)$$
where the sup is w.r.t. all future directed causal curves $\gamma$ that connect  $(p,x)$ and $(q,y)$.  If there is no such curve, one sets $\tau_\slgc((p,x), (q,y))=0$.

 The time separation function $\tau_{\slgc}$  satisfies the reverse $\triangle$-inequality 
$$\tau_{\slgc}((p,x), (q,y)) \geq \tau_{\slgc}((p,x), (r,z))+ \tau_{\slgc}((r,z), (q,y))$$
for all $(p,x), (q,y), (r,z)\in I\times F$. 
Similarly, the signed time separation function $\ell_{\slgc}$ is defined as before.
\begin{definition}[\cite{agks}]
The generalized  cone $\gcl$ between $I$, $f$ and $F$ is  $$(I\times F, \ll, \leq, \tau_{\slwp})$$
where $I\times F$ is the metric  product of $I$ and $F$ equipped with the standard $L^2$-product metric.
\end{definition}
\begin{theorem}[Limit curve theorem, {\cite[Theorem 3.16]{agks}}]\label{th:lim_cur}
Let  $\gamma_i=(\alpha_i, \beta_i):[a,b] \rightarrow \gcl$, $i\in \N\cup \{\infty\}$, be absolutely continuous curves such that  $\gamma_i$ for every $i\in \N$ is future/past directed casual, $\alpha'\neq 0$ $\mathcal L^1$-a.e. and $\gamma_i\rightarrow \gamma_\sinfty$ pointwise. Then $\gamma_\sinfty$ is causal. 
\end{theorem}
\begin{theorem}[Fiber independence, {\cite[Theorem 3.29]{agks}}]\label{th:fiber_ind}
 Let $\gamma=(\alpha,\beta): [0,b] \rightarrow \gcl$  be future directed causal and maximal. 
 \begin{enumerate}
  \item The fiber component $\beta$ is a length minimizer in $X$.
  \item\label{thm-structure-of-geod-fib-ind} Fiber independence holds. The base component $\alpha$ depends only on the length of $\beta$. More precisely, let $(X',\de')$ be another geodesic length space, $\beta'$ minimizing in $X'$ with 
$L^{\de'}(\beta')=L^{\de}(\beta)$ and the same speed as $\beta$, i.e, $v_\beta=v_{\beta'}$. Then $\gamma':=(\alpha,\beta')$ is a 
future directed maximal causal curve in $Y':=I\times_f X'$, which is timelike if $\gamma$ is timelike in $Y$.
 \item If $\gamma$ is timelike and parametrized with respect to arclength, then $v_\beta$ is proportional to 
$\frac{1}{(f\circ\alpha)^2}$. 
\item If $\gamma$ is timelike, then it has an (absolutely continuous) parametrization with respect to arclength, i.e., 
$-\dot\alpha^2 + (f\circ\alpha)^2 v_\beta^2 = -1$ almost everywhere. 

 \end{enumerate}
\end{theorem}

\begin{theorem}[Properties of generalized cones]\label{Pr: various properties}
    Let $Y=-\!I\times_fF$ be a generalized cone as described above and $(F,\de)$ a metric space.
    \begin{itemize}
        \item[(i)] If $F$ is a geodesic length space 
        every maximizing causal curve $\gamma = (\alpha,\beta): [-b,b]\to Y$ has a causal character,
        i.e., $\gamma$ is either timelike or null \textnormal{\cite[Cor.\ 3.30]{agks}}.
        \item[(ii)] If $F$ is locally compact  the length $L^{\slgc}$ coincides with the $\tau$-length $L^{\tau}$ on future-directed causal curves  \textnormal{\cite[Prop.\ 4.7]{agks}}.
        \item[(iii)] If $F$ is a geodesic and locally compact length space then $Y$ is a strongly causal and regular Lorentzian length space \textnormal{\cite[Cor.\ 4.9]{agks}}.
        \item[(iv)] If $F$ is a locally compact, complete length space then $Y$ is globally hyperbolic \textnormal{\cite[Cor.\ 4.11]{agks}}.
        \item[(v)] If $X$ is geodesic, then $Y$ is geodesic too. Furthermore, any two timelike related points can be connected by a timelike geodesic \textnormal{\cite[Cor.\ 4.11]{agks}}.
        In particular, we have that the time separation function $\tau_{\slgc}$ is continuous under these assumptions. 
    \end{itemize}
\end{theorem}
\begin{remark}\label{rem:trans_cone}
Let $I$ and $I'$ be two bounded, open intervals. There exists a smooth  montone bijection $\varphi: I'\rightarrow I$ that is a Riemannian  isometry. 
Then,  $-I\times_f F$ and $-I'\times_{f'} F$ are isomorphic in the sense that the isomorphism  $\phi: -I'\times_{f'} F\rightarrow -I\times_{f} F$  given by $\phi(s,x)= (\varphi(s), x)$ is distance and $\ell$ preserving. 
\end{remark}
\begin{remark}
Let $-I\times_f X$ be a generalized cone for  $f$ and a geodesic metric space $X$.  If $\lambda>0$ is a constant, we can consider $\lambda \cdot f= f_\lambda$ and $\frac{1}{\lambda} X= X_{1/\lambda}$. Then, $\Phi(s,x)=(s,x)$ is an isomorphism between $-I\times_f X$ and $-I\times_{f_\lambda} X_{1/\lambda}$. 
\end{remark}
\subsubsection{Generalized $N$-cones}
We can choose  $F$ as a metric measure space, i.e. a metric space $(F, \de_F)$   equipped with locally finite Borel measure $\m_F$. We fix a parameter $N\in [0, \infty)$. Then, we define the measure $f^N \de t \otimes \de \m_F$ on $-I\times_f F$. We write $-I\times_f^\sN F$ for $-I\times_f F$ equipped with $f^\sN \de t \otimes \de \m_F$, and we call $-I\times_{f}^\sN F$ the generalized $N$-cone between $-I, f$ and $F$. The notion is motivated from related definitions for metric measure spaces \cite{ketterer1, bastco}.
\begin{remark} If the reference measure $\m$ is the $N$-dimensional Hausdorff measure $\mathcal H^N$, we expect that $f^N\de t \otimes \de\mathcal H^N$ is the $(N+1)$-dimensional Lorentzian Hausdorff measure in the sense of \cite{mcsa}.
\end{remark}
\subsection{Generalized cones over smooth Riemannian manifolds}
Let $M$ be a complete Riemannian manifold.  Here $g_M$ denotes the Riemannian metric. 
The  (Riemannian) warped product  $I\times_f M$ between $I$, $f$ and $M$ is  $I\times M$ equipped with $C^0$ Riemannian metric
$$ g_{\scriptscriptstyle I\times_{f} M}=(\de t)^2 + f^2 g_M.$$

The Lorentzian warped product  $- I\times_f M$ between $I$, $f$ and $M$ is  $I\times M$ equipped with the $C^0$ semi-Riemannian metric
$$ g_{\scriptscriptstyle -I\times_{f} M}=-(\de t)^2 + f^2 g_M.$$
If $f$ is smooth then $-I \times_f M$ is a smooth Lorentzian manifold.

For a Lipschitz continuous curve $\beta:[a,b] \rightarrow M$ the metric speed $|\beta'|$ coincides with $\sqrt{g_M(\beta', \beta')}$. 
Hence, if $\gamma:[a,b]\rightarrow  I \times M$ is an admissible path, then 
$$g_{\scriptscriptstyle -I \times_f M}(\gamma', \gamma')= -(\alpha')^2  + f^2\circ \alpha |\beta'|^2.$$

The semi-Riemannian warped product $-I \times_f M$  is  a generalized cone.

The Lorentzian volume form of $g_{\scriptscriptstyle -I\times_f M}$ on $I\times M$ is exactly the measure $f^N \de t \otimes \de \vol_{g_M}$.
\section{$\ell$-convergence of covered Lorentzian pre-length spaces}\label{sec:ell}
\begin{definition}[Covered metric spaces]\label{def:covering}
Let $(X, \de_\sX)$  be a metric space. A countable cover $\mathcal U= \{U^k\}_{k\in \N}$ of $X$ is a family of subsets such that 
\begin{enumerate}
\item  $ \bigcup_{k\in \N} U^k = X$, 
\item $U^k \subset U^{k+1}$ $\forall k \in \N$, 
\end{enumerate}
If  $U_k$ is relatively compact for all $k\in \N$, we call $\mathcal U$ proper. A metric space $X$ equipped with a proper cover is called properly covered. 
\end{definition}
\begin{definition}[Covered Lorentzian pre-length spaces]
A covered Lorentzian pre-length space $Y$ is a  Lorentzian pre-length spaces $(Y, \de, \ell)$ with a countable cover $\mathcal U$ such that $\forall k\in \N$
\begin{enumerate}
\item[(3)] 
$\sup_{x, y\in U^k} \tau(x,y) < \infty$.
\end{enumerate}
We say the cover is weak causal convex if
\begin{enumerate}
\item[(4)] $\forall x \leq y \in U^k$: $J(x,y)\subset U^{k+1}$. 
\end{enumerate}
 A Lorentzian pre-length space with a proper cover is called a properly covered Lorentzian pre-length space. 
\end{definition}

\begin{definition}[Covered GH convergence]\label{def:coveredGH}
Consider a sequence of properly covered  Lorentzian pre-lengthspaces ${Y}_i= (Y_i, \de_i, \ell_i, \mathcal U_i)$, $i\in \N$. 

We say that ${Y}_i$  converges in {\it covered GH sense} to a properly covered Lorentzian pre-length space ${Y}= (Y, \de, \ell, \mathcal U)$ if  $$(\overline U^k_i,  \de_i|_{\overline U^k_i\times \overline U^k_i}) \overset{GH}{\longrightarrow} (\overline U^k,  \de|_{\overline U^k \times \overline U^k}) \ \ \forall k\in \N.$$

In particular, 
for every $k\in \N$ 
there is a complete metric space $Z^k$ and  there are isometric embeddings $\iota^k_i: \overline U^k_i \rightarrow Z^k$, $\iota^k: \overline U^k \rightarrow Z^k$ such that 
$$ \iota_i^k(\overline U_i^k)\rightarrow \iota_i^k(\overline U^k) \mbox{ in Hausdorff sense}.$$
\end{definition}

\begin{definition}[Uniformly non-totally imprisoning]\label{def:uni} Let $\{Y_i\}_{i\in I}$ be a family of properly covered Lorentzian pre-length spaces. We say $\{Y_i\}_{i\in I}$ is uniformly non-totally imprisoning if for every $k\in \N$ there exists a constant $C(k)>0$ such that for $i\in \N$ and   for a causal curve $\gamma$ in $U^k_i$ the $\de_i$-arclength is bounded from above by $C(k)$.
\end{definition}
\begin{definition}[Uniform convergence of signed time separation functions]\label{def:uniform_convergence}
Let  ${Y}_i=(Y_i, \de_i, \ell_i, \mathcal U_i)_{i\in \N}$ be a family of properly covered Lorentzian pre-length spaces  converging in covered GH sense to  a properly covered Lorentzian pre-length space ${Y}=(Y, \de, \ell, \mathcal U)$.

We say that the time separation functions $\ell_i$ {\it converges uniformly} to $\ell$ if  for $k, l\in \N$ and $\epsilon\in \left(0, \frac{1}{2l}\right)$ there exists  $\delta(k,l,\epsilon)>0$ and $i(k,l,\epsilon)\in \N$ such that for   $\delta\in (0, \delta(k,l,\epsilon))$ and $i\geq i(k,l,\epsilon)$ we have $\de_{GH}\big(\ol U^k_i, \ol U^k\big)\leq \delta$ and for a given  $\delta$-GH-isometry $\phi: \overline U^k_i\rightarrow \overline U^k$  the following two properties hold:
\begin{enumerate} 
\item if $(x_i, y_i)\in  \left\{ \ell_i \geq 0 \right\}\cap \ol U_i^{k}\times \ol U_i^{k}$, then 
$$\ell_i(x_i,y_i)\leq \ell(x,y)+ \epsilon$$
 for every $(x,y) \in \ol U^{k}\times \ol U^{k}$ such that $ \de(\phi(x_i), x) +\de(\phi(y_i), y)\leq \delta$.
\item 
 if $ (x, y)\in \left\{ \ell\geq   \frac 1 \el \right\} \cap \overline U^{k} \times \ol U^{k}$, then $$\ell_i(x_i, y_i) \geq \ell(x,y)-\epsilon$$ 
for every   $(x_i, y_i) \in \ol U_i^{k} \times \ol U_i^{k}$ such that $\de(\phi(x_i), x) +\de(\phi(y_i), y)\leq \delta$.
\end{enumerate} \end{definition}
\begin{remark}
The following modification is equivalent to the previous definition. If $Y_i$ converges in covered GH sense to $Y$, then for every $k\in \N$ we find a compact metric space $Z^k$ such that $\ol U^k_i$ and $\ol U^k$ embed distance preservingly into $Z^k$ and $\ol U^k_i\overset{H}{\rightarrow}\ol U^k$, i.e. convergence in Hausdorff sense. Then, we can replace the condition  {\it for every $(x,y) \in \ol U^{k}\times \ol U^{k}$ such that $ \de(\phi(x_i), x) +\de(\phi(y_i), y)\leq \delta$} where $\phi$ is $\delta$-GH-isometry, with  {\it for every $(x,y) \in \ol U^{k}\times \ol U^{k}$ such that $ \de^{Z^k}(x_i, x) +\de^{Z^k}(y_i, y)\leq \delta$} where the corresponding $\delta(\epsilon, k, l)$ may be different. 
\end{remark}
\begin{remark}\label{rem:uniform_convergence}
Let $Z=Z^{k}$ be a compact metric space s.t.  $\overline U^{k}_i\overset{H}{ \longrightarrow}\overline U^{k}_\sinfty$  in $Z$. 
\\
{\it Claim:}  Let $\epsilon \in \left(0, \frac 1 {2\el}\right)$, and $i(k, l, \epsilon)$  and $\delta(\epsilon, k, l)$ as in the definition.  Then $$\left\{ \ell_i \geq \frac 1 \el - \epsilon\right\} \subset B_\delta\left(\left\{ \ell\geq \frac 1 \el {\color{black} - 2\epsilon}\right\}\right)\  \&  \  B_\delta\left(\left\{ \ell\geq \frac 1 \el\right\}\right) \cap \ol U_i^k\subset \left\{ \ell_i \geq \frac 1 {2\el}\right\}$$ for $i\geq i(k,\el, \epsilon)$ and $\delta \in (0, \delta(k, \el, \epsilon))$.

Moreover, if $(x,y)\in \{ \ell \geq \frac 1 \el\}$ and $(x_i, y_i)\in \overline U^k_i$ s.t. $\de^Z(x_i, x) + \de^Z(y_i, y)\leq \delta$, then
$$\left| \ell_i (x_i, y_i) - \ell(x,y)\right| \leq \epsilon.$$
\smallskip
{\it Proof of the Claim:}  
For the first inclusion we consider $(x_i, y_i)\in \left\{ \ell_i\geq \frac 1 \el - \epsilon \right\}$.  We find $(x,y)\in \ol U^{k}_\sinfty\times \ol U^{k}_\sinfty$ such that  $\de^Z(x_i, x)+ \de^Z(y_i,y)\leq \delta$. Hence 
$$ \ell_i(x_i, y_i) \leq \ell(x,y)+ \epsilon.$$ 
Consequently,   $(x,y) \in \left\{ \ell \geq \frac 1 \el {\color{black} - 2\epsilon} \right\}$. 

For the second inclusion we pick $x_i,y_i \in \overline U^k_i$. We observe that
$$(x_i,y_i)\in B_\delta\left(\left\{\ell\geq  \frac 1 \el\right\}\right) \  \Leftrightarrow \  \exists (x,y)\in  \left\{\ell\geq \frac 1 \el\right\}:  \de^Z(x,x_i)+ \de^Z(y,y_i)<\delta.$$
Hence 
$$\ell_i(x_i, y_i)\geq \ell(x,y)- \epsilon \geq \frac 1 {2\el}.$$

Finally, if $(x,y)\in \{\ell\geq \frac 1 \el\}$ and $(x_i, y_i)\in \overline U^k_i$ s.t. $\de^Z(x_i, x)+ \de^Z(y_i, y)\leq \delta$, then $(x_i, y_i)\in \{ \ell_i \geq \frac 1 {2 \el}\}\subset \{ \ell_i \geq 0\}$.
Hence 
$$\left| \ell_i (x_i, y_i) - \ell(x,y)\right| \leq \epsilon.$$
for $i\geq i(\epsilon, \el)$ and $(x_i, y_i)\in Z$ such that $\de^Z(x_i, x)+ \de^Z(y_i, y)\leq \delta$.
\hfill ${\scriptstyle \square}$
\end{remark}
\pagebreak[2]
\begin{definition}[$\ell$-convergence]\hspace{0mm}\label{def:ellconvergence}
\begin{enumerate}
\item
We say that a sequence $Y_i$ of  Lorentzian pre-length spaces {\it $\ell$-converges} to a  Lorentzian pre-length space $Y_\sinfty$ if for every $i\in \N\cup\{\infty\}$ there exist proper covers  $\{U^k_i\}$ on $Y_i$ such that $Y_i$ is a properly covered Lorentzian pre-length space, and the following properties hold:
\begin{enumerate}
\item $(Y_i)_{i\in \N}$ converges in covered GH sense to $Y$, 
\item $\{Y_i\}_{i\in \mathbb N}$ is uniformly non-totally imprisoning, 
\item $\ell_i$ converges uniformly to $\ell$. 
\end{enumerate}
We write $Y_i \overset{\ell}{\longrightarrow} Y$. 
\item
A sequence of pointed Lorentzian pre-length spaces $(Y_i, o_i)$ converges in pointed $\ell$-sense to a pointed Lorentzian pre-length space $(Y_\sinfty, o_\sinfty)$ if in addition to the previous properties we have 
\begin{enumerate}
\item $o_i\in U^k_i$ for all $k\in \N$ and for every $i\in \N\cup\{\infty\}$, 
\item $o_i\in U^k_i\rightarrow o_\sinfty\in U^k_\sinfty$ $\forall k\in \N$, 
\end{enumerate}We write $(Y_i,o_i) \overset{p\ell}{\longrightarrow} (Y, o)$. 
\end{enumerate}
\end{definition}
Since $\ell$-convergence mimics the uniform convergence of functions along GH-converging sequences, it may not be surprising that it preserves continuity of the time separation functions in the limit. 
\begin{lemma}\label{lem:cont_tau}
Let $Y_i$ be properly convered Lorentzian pre-length spaces {\it converging} to  a properly covered Lorentzian pre-length space $Y_\sinfty$. If $\tau_i$ is continuous, then $\tau_\sinfty$ is continuous. 
\end{lemma}
\begin{proof} We recall that $\tau_i = \max\{0, \ell_i\}$, $i\in \N\cup\{\infty\}$.  
By the previous remark we have that $\ell_\sinfty$ restricted to  $\{\ell_\sinfty\geq \frac 1 \el\} \cap \ol U^k_\sinfty$ is the uniform limit of continuous functions. Hence $\ell_\sinfty|_{\{\ell_\sinfty>0\}\cap \ol U^k_\sinfty}$ is continuous. 

Let $(x,y)\in Y_\sinfty\times Y_\sinfty \cap \ol U^k_\sinfty$ for some $k\in \N$ and $\ell_\sinfty(x,y)\leq 0$. It follows that $\tau_\sinfty(x,y)=0$.  Let $(x^j , y^j)$ be sequence that converges to $(x,y)$.  We will show that $\tau_{\sinfty}(x^j, y^j)\rightarrow \tau_\sinfty(x,y)=0$. For this we argue by contradiction. We assume that there exists $\epsilon>0$ such that $\ell_\sinfty(x^j, y^j)\geq \epsilon$ for all $j\in \N$.   There exist $x^j_i, y^j_i\in \ol U_i^k$ such that $x^j_i\rightarrow x^j$ and $y^j_i\rightarrow y^j$.  For $i$ sufficiently large, it holds that $(x^j_i, y^j_i)\in \{ \ell_i \geq \epsilon/2\}$.  A diagonal sequence $(x^{j_i}_i, y^{j_i}_i)$ converges to $(x,y)$. On the other hand, $\ell_i(x^{j_i}_i, y^{j_i}_i)\geq \epsilon/2$ for all $i\in \N$. Hence $\ell_\sinfty(x,y)\geq \epsilon/2>0$.
\end{proof}
\begin{lemma} We consider properly covered, weak causal convex Lorentzian geodesic spaces $Y_i$ such that $Y_i\overset{\ell}{\longrightarrow} Y$ for a properly covered Lorentzian pre-length space $Y$. Then 
$$\Gamma_i^k= \left\{ \gamma \in \Geo(Y_i): \gamma(0), \gamma(1)\in \ol U^k_i \mbox{ and } \Im \gamma\subset \ol U^{k+1}_i\right\}$$ is pre-compact w.r.t. uniform convergence and every limit $\gamma$ is in $\Geo(Y)$ such that $\gamma(0), \gamma(1)\in \overline U^k_\sinfty$ and $\mbox{Im} \gamma\subset \overline U_\sinfty^{k+1}$. 
\end{lemma}
\begin{proof}
By assumption the cover $\mathcal U_i$ is weak causal convex.  
Since  $\{Y_i\}_{i\in \N}$ is uniformly non-totally imprisoning, it  follows that every  $\gamma\in \Gamma^k_i$ has $\de$-length bounded by $C(k+1)$. For every $i\in \N$ the set $\ol U^{k+1}_i$ embeds into a compact metric space $Z^{k+1}$. It follows by the Arzela-Ascoli Theorem that a sequence  $(\gamma_i)_{i\in \N} $ of  curves $\gamma_i \in \Gamma_i^k$ sub-converges w.r.t. the uniform distance $\de^\sinfty$ to a continuous curve $\gamma:[0,1] \rightarrow Z^{k+1}$ such that $\gamma(0), \gamma(1)\in \ol U_\sinfty^k$ and $\Im \gamma\subset \ol U^{k+1}_\sinfty$.  Recall that $\de^\infty(\gamma^0, \gamma^1)= \sup_{t\in [0,1]} \de^{Z^{k+1}}(\gamma^0(t), \gamma^1(t))$ for two continuous curves $\gamma^0, \gamma^1:[0,1] \rightarrow Z^{k+1}.$
\smallskip\\
{\it Claim: $\gamma\in \Geo(Y_\sinfty)$}.
{\it Proof of the claim: }
From property (1) in the Definition \ref{def:uniform_convergence}  it follows that $\ell_\sinfty(\gamma(0), \gamma(1))\geq 0$. If $\ell_\sinfty(\gamma(0), \gamma(1))\geq \frac 1 \el$ for some $\el\in \N$, we also obtain that $\tau(\gamma(s), \gamma(t))= (t-s) \tau(\gamma(0), \gamma(1))$ for all $s< t$ in $[0, 1]$ by property (2) in Definition \ref{def:uniform_convergence} and Remark \ref{rem:uniform_convergence}. It follows that $\gamma\in \Geo(Y_\sinfty)$. \hfill $\triangle$
\smallskip\\
If $\ell_\sinfty(\gamma(0), \gamma(1))=0$, it follows that 
$$\limsup_{i\rightarrow \infty} \ell_i(\gamma_i(0), \gamma_i(0))\leq 0$$
by property (1) of uniform convergence of $\ell_i$. Therefore, $\ell_i(\gamma_i(0), \gamma_i(1))\rightarrow 0$ and also $\ell_i(\gamma_i(s), \gamma_i(t))\rightarrow 0$ for all $s, t\in [0,1]$. It follows that $\ell_\sinfty(\gamma(s), \gamma(t))=0$ for all $s, t\in [0, 1]$, and consequently $\gamma\in \Geo(Y_\sinfty)$ also in this case. 
\end{proof}
\subsubsection{Measured $\ell$-convergence} Let $Y_i$, $i\in \N\cup\{\infty\}$, 
be properly covered, measured Lorentzian pre-length spaces  such that $Y_i$    converges in covered GH sense to $Y_\sinfty= Y$ for $i\rightarrow \infty$.  Let $\{U^k_i\}$ be associated coverings. We consider  embeddings $\iota_\sinfty^k, \iota_i^k: \overline U_\sinfty^k, \overline U^k_i \rightarrow Z^k$ for a metric space $Z^k$, and  set $\m_i(\overline U^k_i)^{-1} \m_i|_{\overline U^k_i}= \overline \m_i^k$, $i \in \N\cup\{\infty\}$. 
We say $\m_i$ converges weakly to $\m_\sinfty$ if  for every $k\in \N$ we have that $(\iota_i^k)_\sharp \ol \m_i^k$ converges weakly to $(\iota_\sinfty^k)_\sharp \ol \m_\sinfty^k$ in $Z^k$. 
We say a sequence of measured Lorentzian pre-length spaces $(Y_i)_{i\in \N}$ converges in measured $\ell$-sense to a measured Lorentzian pre-length space $Y$, we write $Y_i \overset{\m \ell}{\longrightarrow} Y$, if $Y_i$ $\ell$-converges to $Y$ and $\m_i$ converges weakly to $\m_\sinfty$. 
\begin{remark}
Measured $\ell$-convergence essentially covers the type of convergence that was considered by Cavalletti Mondino in \cite{camolorentz} to show stability of timelike curvature-dimension conditions. Roughly, in \cite{camolorentz} it was said that a sequence of measured Lorentzian geodesic spaces $Y_i$ converges to a measured Lorentzian space $Y_\sinfty$ if there exists a Lorentzian space $\bar Y$ such that all spaces isomorphically embed into $\bar Y$ and the measures $\m_i$ converges weakly to $\m_\sinfty$ inside of $\bar Y$. In our situation this corresponds to the case when the sequence $Y_i$ consists of identical copies of one Lorentzian geodesic space and only the measures vary along the sequence. 
\end{remark}
It will make sense to have the following local versions of the curvature-dimension conditions and of the measure contraction property. 
\begin{definition} Let $K\in \R$, $N\in (0, \infty)$ and $k\in \N$. Let $Y$ be properly covered Lorentzian pre-length space and let $\mathcal U=\{U^k\}_{k\in \N}$ be the associated cover. 
We say $Y$ satisfies the entropic timelike curvature-dimension condition $\TCD_p^e(K,N)$ locally in $U^k$ if 
for every timelike $p$-dualizable pair $\mu_0,\mu_1 \in D(\Ent_\m)\cap \mathcal P_c(Y)$ with $\mu_i(U^k)=1$, $i=0,1$, there exists an ${\ell_p}$-geo\-desic $(\mu_t)_{t\in [0,1]}$ connecting $\mu_0$ and $\mu_1$  as well as a timelike $p$-dualizing  coupling $\pi\in{\Pi^{p-opt}_\ll(\mu_0,\mu_1)}$ such that $\forall t\in [0,1]$,
\begin{align*}
U_N(\mu_t) \geq \sigma_{K,N}^{(1-t)}\big[\left\| \tau \right\|_{L^2(\pi)}\big]U_N(\mu_0) + \sigma_{K,N}^{(t)}\big[\left\| \tau \right\|_{L^2(\pi)}\big]U_N(\mu_1).
\end{align*}
 If the previous claim holds for every strongly timelike $p$-dualizable  $\mu_0,\mu_1\in D(\Ent_\m)\cap \mathcal P_c(Y)$,  $Y$ satisfies the \emph{weak entropic timelike curvature-dimension condition} ${ w\TCD_p^e(K,N)}$.
 
 Analogously, we define the (reduced, weak) timelike curvature-dimension condition $\TCD_p(K,N)$ ($\TCD_p^*(K,n)$, $w\TCD_p^e(K,N)$) and the timelike measure contraction property $\TMCP(K,N)$ locally in $U^k$. 
 \end{definition}

\section{Stability of curvature bounds under $\ell$-convergence}\label{sec:stability}
\subsection{Stability of timelike curvature-dimension conditions}
\begin{theorem}\label{th:stabilityTCD}
Let $(Y_i, \de_i, \ell_i, \mathcal U_i, \m_i)_{i\in \N}$ be a sequence of properly covered, measured Lorentzian geodesic spaces that satisfy the condition $\TCD_p(0,N)$ for some $p\in (0,1)$. We assume $(Y_i, \de_i, \ell_i, \mathcal U_i, \m_i)_{i\in \N}$ converges in the measured  $\ell$-sense to a properly covered, measured Lorentzian geodesic space $(Y, \de, \ell, \mathcal U, \m)$.  Let $\mathcal U_i$ be weak causal convex $\forall i\in \N$. 

Then $(Y, \de, \ell, \mathcal U, \m)$ satisfies the weak timelike curvature-dimension condition $w\TCD_p(0,N)$. 

The same statments hold with $\TCD_p^e(K,N)$ in place of $\TCD_p(0, N)$ for any $K\in \R$ and $N\in [1, \infty)$. 
\end{theorem}
\begin{remark} \begin{enumerate}
\item We expect that the  condition 
$\TCD_p(K,N)$ and the condition $\TCD_p^*(K,N)$  that were introduced in \cite{braun_renyi} obey the same stability properties under $\m\ell$-convergence of measured Lorentzian geodesic spaces. Proofs of corresponding stability results are technically more involved but  should work along the same lines as in \cite{braun_renyi} (compare also with \cite{stugeo1}) taking into account the  details of the proof below. 
\item It will be become clear from the proof that the following stronger statement for the curvature-dimension condition locally in $U^k$ holds:\\
{\it 
Let $(Y_i)_{i\in \N}$ be a sequence of properly covered Lorentzian geodesic spaces with weak causal convex covers such that $(Y_i)$ converges in the measured $\ell$-sense to a  properly covered, measured Lorentzian geodesic spaces $Y$. We assume that for all $i\in \N$ the condition $\TCD_p(0,N)$ holds locally in $U^k_i$ for $k\in \N$ and $p\in (0,1)$. 
Then the weak condition $w\TCD_p(0,N)$ holds locally in $U^k$. 
An analogous statment holds for the conditions $\TCD_p^e(K,N)$ in place of $\TCD_p(0,N)$ for any $K\in\R$ and $N\in [1,\infty)$. }
\end{enumerate}
\end{remark}
\begin{proof} {\bf \large (0)} We show the statement only for the condition $\TCD_p(0,N)$. For the condition $\TCD_p^e(K,N)$ we will see in the end what are the modifications needed in the proof.
\smallskip\\
{\bf \large (1)} 
We fix $\mu^{\sinfty}_0, \mu^{\sinfty}_1\in \mathcal P_c(Y_\sinfty, \m_\sinfty)$ that are strongly $p$-dualizable, i.e.  there exists  $\pi^\sinfty \in \Pi_{\leq}^{p-opt}(\mu^{\sinfty}_0, \mu^{\sinfty}_1)$ such that $\pi^\sinfty(\{\ell_\sinfty>0\})=1$,  and  there exists a measurable, $\ell_\sinfty^p$-cyclically monotone set $\Gamma\subset \{\ell_\sinfty>0\} \cap (\spt \mu_0^\sinfty\times \spt \mu_1^\sinfty)$ such that  any coupling $\pi \in \Pi_{\leq}(\mu_0^\sinfty, \mu_1^\sinfty)$ is $\ell_\sinfty^p$-optimal if and only if $\pi(\Gamma)=1$.

We fix $k\in \N$ and we assume that $\pi^\sinfty\left( \overline{U}_\sinfty^{k}\times \overline U_\sinfty^{k}\right)=1$. To simply the notation we will omit the index $k$ in the following steps.
\medskip\\
{\bf \large (2)}
We find couplings $\pi^{\sinfty, n}$ such that for $n\rightarrow \infty$ we have
\begin{enumerate}
\item $\pi^{\sinfty, n} \rightarrow \pi^\sinfty$ weakly in $\overline U_\sinfty\times \ol U_\sinfty$, 
\item $(P_j)_\sharp \pi^{\sinfty, n} = \mu_j^{\sinfty,n} \rightarrow  \mu^\sinfty_j$ weakly in $\ol U_\sinfty$, $j=0, 1$,
\item $\pi^{\sinfty, n} = \rho^{\sinfty, n} \m_\sinfty \otimes \m_\sinfty$ with $\rho^{\sinfty, n} \in L^\sinfty(\m_\sinfty\otimes \m_{\sinfty}), $
\item $\pi^{\sinfty, n}( \{\ell_\sinfty>0\})=1$, 
\end{enumerate}Moreover \begin{align}\label{ei1}
S_N( \mu_j^{\sinfty, n} | \m_\sinfty) \rightarrow S_N(\mu^\sinfty_j| \m_\sinfty)\mbox{ for $j=0,1$. }\end{align}
\smallskip
\noindent
{\bf \large (3)} 
We fix $n\in \N$ for a moment.  Then $$\pi^{\sinfty,n}\left(\left\{\ell_\sinfty\geq \frac 1 \el\right\}\right)=: a^\el \rightarrow 1 \mbox{ for } \el\in \N\rightarrow \infty.$$We set 
$$(\pi^\el)^{\sinfty, n}= (a^\el)^{-1} \pi^{\sinfty, n}|_{\left\{\ell_\sinfty \geq  \frac 1 \el\right\}}$$ and 
$(\mu^\el)^{n, \sinfty}_j:= (P_j)_\sharp (\pi^\el)^{\sinfty, n}$ for $j=0,1$.

We have that $(\pi^\el)^{\sinfty, n} \in \Pi_{\leq}^{p-opt}\left((\mu^\el)_0^\sinfty, (\mu^\el)_1^\sinfty\right)$ and $(\pi^\el)^{\sinfty, n}$ satisfies
\begin{enumerate}
\item $(\pi^\el)^{\sinfty, n} \rightarrow \pi^{\sinfty, n}$ weakly for $\el \rightarrow \infty$,
\item $(\mu^\el)_{j}^{\sinfty, n}\rightarrow \mu_{j}^{\sinfty,n}$ weakly for $\el \rightarrow \infty$, $j=0,1$,
\item $(\pi^\el)^{\sinfty, n} = (\rho^\el)^{\sinfty, n} \m_\sinfty \otimes \m_\sinfty$ with $(\rho^\el)^{\sinfty, n} \in L^\sinfty(\m_\sinfty\otimes \m_{\sinfty}), 
$\item $(\pi^\el)^{\sinfty, n}\left( \left\{\ell_\sinfty\geq \frac 1 {\el}\right\}\right)=1$.
\end{enumerate}
\smallskip
{\bf (3.1)}
Since $x\in [0, \infty) \mapsto v(x)= -x^1-\frac 1 N$ is convex with $v(0)=0$, it holds that 
$$u(x+y)\geq u(x)+u(y) \ \forall x,y \geq 0.$$
Moreover 
$$(\mu^l)_{j}^{\sinfty,n}= (a^\el)^{-1} \mu_{j}^{\sinfty,n}= (a^\el)^{-1} \rho_{j}^{n,\sinfty} \m_\sinfty$$
where $0\leq ( a^\el)^{-1} \rho_{j}^{\sinfty,n} \leq (a^\el)^{-1} \left\| (\rho^\el)^{\sinfty, n}\right\|_{L^\sinfty}.$

Then, by an argument as in Step 2b of the proof of Theorem 3.15 in \cite{camolorentz} using concavity of $x^{1-\frac 1 N}$, it follows
\begin{align}\label{ei2} \limsup_{\el \rightarrow \infty} S_N((\mu^\el)_{j}^{\sinfty,n}|\m_\sinfty)\leq 
 S_N(\mu_{j}^{\sinfty, n}|\m_\sinfty). \end{align}

\medskip
\noindent
{\bf \large (4)}
Recall that $\ol U_i\overset{H}{\longrightarrow} \ol U_\sinfty$ in $Z$.  We continue to omit the superscript $k$. 

Moreover, $\m_i|_{\overline{U}_i}$ converges weakly to $\m_\sinfty|_{\overline{U}_\sinfty}$ in $Z$. Since $Z$ is compact, this implies $W^{\de_Z}_2$-Wasserstein convergence in $Z$. 

Hence, there are $\de_Z^2$-optimal couplings $\boldsymbol{p}^i$  between the normalized measures $\m_\sinfty(\ol U_i)^{-1} \m_\sinfty|_{\ol U_\sinfty}=: \bar \m_\sinfty$ and $\m_i(\ol U_i)^{-1} \m_i|_{\ol U_i}= :\bar \m_i$. 

We use ${\boldsymbol p}^i$ to build a coupling $ \pi^{i, n}$ between probability measures in $\overline U_i$ as follows. One defines $(P_{2,4})_\sharp \tilde \pi^{\el, i, n} = \pi^{\el, i,n}$ where 
$$\tilde \pi^{\el,i,n}(\de x_1, \de x_2, \de x_3, \de x_4)= (\rho^\el)^{\sinfty, n}(x_1, x_3) \boldsymbol{p}^i( \de x_1, \de x_2) \otimes \boldsymbol{p}^i(\de x_3, \de x_4).$$
It holds that 
\begin{enumerate}

\item $  \pi^{\el, i,n} \rightarrow (\pi^\el)^{\sinfty, n}$ weakly in $Z^2$, 
\item $(P_j)_\sharp \pi^{\el,i,n}=:\mu_{j}^{\el, i,n}\rightarrow (\mu^\el)_{j}^{\sinfty,n}$ weakly in $Z$ as $i\rightarrow \infty$, $j=0,1$,\item $\pi^{\el, i,n} \ll \bar \m_i \otimes \bar \m_i$.
\end{enumerate}
The measure $\mu_j^{\el, i,n}$ is $\m_i$-absolutely continuous, and its density satisfies $\rho^{\el, i,n}_j \leq (\rho^\el)_{j}^{\sinfty,n}\leq \rho_{j}^{\sinfty,n}.$ 
\medskip\\
{\bf (4.1)}
A standard argument, that appears in  \cite{stugeo1, stugeo2}  and in \cite{lottvillani}  (see also Step 2.1 of the proof of Theorem 3.15 in \cite{camolorentz}), using Jensen's inequality, yields
\begin{align}\label{ei3}S_N(\mu_j^{\el, i,n}|\m_i)\leq S_N((\mu^\el)_{j}^{\sinfty,n}|\m_\sinfty).\end{align}
\smallskip
\noindent
{\bf \large (5)}
Let $\epsilon \in (0, \frac{1}{k})$, and choose $i(\epsilon, k)$ and $\delta(\epsilon, k)$ as in the definition of uniform convergence. Let $\delta\in (0, \delta(\epsilon, k))$. 

Since $B_\delta(\{\ell_\sinfty\geq \frac 1 l\})$ is open, it follows that 
\begin{align*}
\liminf_{i\rightarrow \infty} \pi^{\el, i,n}(B_\delta(\{\ell_\sinfty\geq \frac 1 l\} ))& \geq (\pi^\el)^{\sinfty, n}(B_\delta(\{\ell_\sinfty\geq \frac 1 l\}))\\
&= (\pi^\el)^{\sinfty, n}(\{\ell_\sinfty\geq \frac 1 l\})=1.\end{align*}
Since (Remark \ref{rem:uniform_convergence})
$$1\geq \underbrace{ \pi^{\el,i,n}\left(\left\{ \ell_i \geq \frac 1 {2\el}\right\}\right)}_{=:b^i}\geq  \pi^{\el, i,n} \left(B_\delta\left(\Big\{\ell_\sinfty\geq \frac 1 \el\Big\}\right)\right),$$
it follows that $b^i\rightarrow 1$. 
Then, we define 
$$(\pi')^{\el, i, n}= (b^i)^{-1} \pi^{\el, i,n}|_{\left\{\ell_i\geq \frac 1 {2\el}\right\}}.$$We set $(P_j)_\sharp (\pi')^{\el, i,n}=: ( \mu')_j^{\el, i,n}$, $j=0,1$. 
Hence
 \begin{enumerate} 
 \item $(\pi')^{\el, i,n}$ is concentrated in $\{\ell_i\geq \frac{1}{2\el}\}$, 
\item  $(\pi')^{\el, i,n}$ converges weakly to $(\pi^\el)^{\sinfty, n}$ in $Z\times Z$ as $i\rightarrow \infty$. 
\item $(\mu')_j^{\el, i,n}$ converges weakly to $(\mu^l)^\sinfty_j$ as $i \rightarrow \infty$.  
\end{enumerate}
\smallskip
{\bf (5.1)}
The measure $(\mu')_j^{\el, i,n}$ is $\m_i$-absolutely continuous and its density $(\rho')_j^{\el, i,n}$ satisfies $0\leq (\rho')^{\el, i,n}_j\leq (b^i)^{-1} \rho_j^{\el, i,n}.$ Hence, as in the previous step, this time not taking a diagonal sequence, it follows
\begin{align}\label{ei4}\limsup_{i\rightarrow \infty} S_N((\mu')_{j}^{\el, i,n}| \m_i)\leq \limsup_{i\rightarrow \infty} S_N(\mu_{j}^{\el, i, n}| \m_i).
\end{align}
\noindent
\\
{\bf  \large (6)} We  address stability of $\ell^p_i$ along a weakly convergent sequence  $\{\pi_i\}_{i\in\N}$. 
\begin{lemma}\label{lem:first_convergence}
Let $\pi_{\sinfty}$ be  such that $\pi_{\sinfty}\left(\left\{ \ell_\sinfty\geq \frac 1 \el\right\}\right)=1$, and let $(\pi_i)_{i\in \N}$ be a sequence of couplings  converging weakly to $\pi_\sinfty$  such that $\pi_i\left( \left\{ \ell_i \geq \frac 1 \el\right\}\right)=1$. 
Then, it holds for $q\in (0,\infty)$ that 
\begin{align*}\lim_{i\rightarrow \infty} \int \ell^q_i \de \pi_i = \int \ell^q_\sinfty \de \pi_\sinfty. 
\end{align*}
\end{lemma}

\begin{proof}[Proof of the Lemma] There exists a metric $\de^{Z}$ on $Z:= \bigsqcup_{i \in \bar \N} \ol U_i$ such that 
\begin{center}  $\ol U_i \times \ol U_i$ converges in Hausdorff sense to $\ol U_\sinfty\times \ol U_\sinfty$  in $Z\times Z$. 
\end{center}
We set $$\Delta=\dot \bigcup_{i\in \N} \left\{\ell_i \geq \frac 1 \el\right\} \dot \cup \left\{ \ell_\sinfty \geq \frac 1 {\el}\right\} \subset Z\times Z.$$
and 
$$\hat \ell: \Delta \rightarrow (0, \infty), \ \ \hat \ell(x,y) = \begin{cases} \ell_i(x,y) & \mbox{ if } (x,y)\in \left\{\ell_i \geq \frac 1 \el\right\}, \\
\ell_\sinfty(x,y) & \mbox{ if } (x,y) \in \left\{ \ell_\sinfty\geq \frac 1 {\el}\right\}.
\end{cases}$$
{\it Claim:} $\hat \ell : \Delta \rightarrow (0, \infty)$ is continuous. 
\smallskip\\
{\it Proof of the Claim: } 
For $\epsilon>0$ 
let $i(\epsilon, \el)$ and $\delta(\epsilon, \el)$ be as in the definition of uniform convergence.

Let $(x,y)\in \left\{ \ell_\sinfty \geq \frac 1 {\el}\right\}$ and let $(x_i,y_i)\in \left\{ \ell_i\geq \frac 1 \el\right\}$ such that $(x_i, y_i)\rightarrow (x,y)$ in $Z\times Z$, i.e. for $\delta \in (0, \delta(\epsilon, \el))$ there exists $i_0\geq i(\epsilon, \el)$ such that  \begin{center}$\de^Z(x_i, x)+ \de^Z(y_i, y)\leq \delta$ for $i\geq i_0$.  \end{center}
It follows that 
$$\left| \ell_i(x,y) - \ell(x,y) \right|\leq \epsilon \mbox{ for } i\geq i_0$$
or 
$\lim_{i\rightarrow \infty} \ell_i(x_i, y_i) =\ell_\sinfty(x,y).$
This proves the claim. \hfill ${\scriptstyle \square}$
\smallskip\\
By construction $\Delta$ is a compact subset of $Z\times Z$.   
Hence, $\hat \ell$ is also bounded. We can extend  $\hat \ell$ as a bounded continuous function to $Z\times Z$. 
Hence
\begin{align*}\lim_{i\rightarrow \infty} \int \ell_i^q \de \pi_i= \lim_{i\rightarrow \infty} \int \hat \ell^q\de\pi_{i}
&= \int \hat \ell^q\de \pi_{\sinfty}= \int \ell_\sinfty^q\de \pi_{\sinfty}. 
\end{align*}
This finishes the proof of the lemma.
\end{proof} 
Lemma \ref{lem:first_convergence} implies that 
\begin{align}\label{ell_conv}\int \ell_i^p \de (\pi')^{\el, i, n}\rightarrow \int \ell_\sinfty^p \de (\pi^\el)^{\sinfty, n} = \ell_{p}( (\mu^\el)^{\sinfty, n}_0, (\mu^\el)^{\sinfty, n}_1) \in (0, \infty). \end{align}
Hence
$\ell_p((\mu')_0^{\el, i,n}, (\mu')_1^{\el, i,n})\in (0, \infty)$, and since $(\pi')^{\el, i, n}$ is concentrated in $\{\ell_i\geq \frac 1 {2\el}\}$ if follows from Remark \ref{rem:lorentz_wasserstein} that there exists 
$$\hat \pi^{\el, i,n}\in \Pi_{\leq}^{p-opt} \left((\mu')_0^{\el, i,n}, (\mu')_1^{\el, i,n}\right) .$$
\smallskip\noindent
{\bf \large (7)}
Now, we define iteratively   diagonal sequences and choose subsequences as follows: 

Firstly,  we consider the diagonal sequence $\left( (\pi^n)^{\sinfty, n}\right)_{n\in \N}$ of $\left( (\pi^\el)^{\sinfty, n}\right)_{\el, n\in \N}$. It satisfies 
\begin{enumerate}
\item $(\pi^n)^{\sinfty, n} \rightarrow \pi^{\sinfty}$ weakly as $n \rightarrow \infty$. 
\item $(\mu^n)_{j}^{\sinfty, n}\rightarrow \mu_{j}^{\sinfty}$ weakly for $n \rightarrow \infty$, $j=0,1$. 
\end{enumerate}
Since $\tau_\sinfty$ is continuous, it follows for any power $q\in (0, \infty)$  that 
\begin{align}\label{tau_conv}\lim_{n\rightarrow \infty} \int \ell_\sinfty^q \de (\pi^n)^{\sinfty, n} = \int \ell_\sinfty^q \de \pi^\sinfty.
\end{align}

Secondly, from Prokhorov's Theorem and since $(\mu')^{n, i, n}_j, j=0, 1$, are compactly supported in $Z$, 
a subsequence $(\hat \pi^{n, i_\alpha,n})_{\alpha\in \N}$ converges weakly to a measure $\hat \pi^{n, \sinfty, n}$. Since $\ell_i$ converges uniformly to $\ell_\sinfty$, we have  that $\hat \pi^{n, \sinfty, n}$  is a coupling between the marginal distributions $(\mu^n)_{j}^{\sinfty,n}$, $j=0,1$   with $\hat \pi^{n, \sinfty, n}(\{\ell_\sinfty\geq 0\}\cap \ol U_\sinfty\times \ol U_\sinfty)=1$. We can iteratively choose subsequences and extract a subsequence $(i_\alpha)_{\alpha \in \N}$ such that  $(\hat \pi^{n, i_\alpha,n})_{\alpha\in \N}$ converges weakly to a measure $\hat \pi^{n,\sinfty,n}$ for all $n\in \N$.  By changing the index we assume that the subsequence $(i_\alpha)_\alpha$ is the sequence $(i)_{i\in \N}$ itself. 
Moreover, again by Prokhorov's Theorem there is a subsequence $(\hat \pi^{n_\beta, \sinfty, n_\beta})_{\beta}$ that converges weakly to $\hat \pi^\sinfty$, a coupling between $\mu^\sinfty_0$ and $\mu^\sinfty_1$ such that $$\hat \pi^\sinfty(\{\ell_\sinfty\geq 0 \}\cap\ol U_\sinfty\times \ol U_\sinfty)=1.$$  Again, we assume that the sequence $(\hat \pi^{n, \sinfty, n})$ already converges.

Thirdly, we consider the diagonal sequence $(\hat \pi^{n, n, n})_{n\in \N}$ of $(\hat \pi^{n, i, n})_{i, n\in \N}$. It satisfies for $n\rightarrow \infty$
\begin{enumerate}
\item $\hat \pi^{n, n, n} =:\hat \pi^n\rightarrow \hat \pi^{\sinfty} \mbox{ weakly}$
\item
$(\mu')_j^{n,n,n} =: \hat \mu_j^n \rightarrow \mu^\sinfty_j \mbox{ weakly, } j=0,1.$
\end{enumerate}
We also note that from \eqref{ell_conv} and \eqref{tau_conv} it follows that \begin{align}\label{ell_conv_2}\lim_{n\rightarrow \infty} \int \ell_n^q \de (\pi')^{n, n, n} = \int \ell^q_\sinfty \de \pi^\sinfty.\end{align}
for any power $q\in (0, \infty)$.
In the following we change the index $n$ back to $i$.
\begin{lemma}\label{lem:optimal_coupling}$\hat \pi^\sinfty\in \Pi_{\geq}^{p-opt}(\mu_0, \mu_1)$.\end{lemma}
\begin{proof}[Proof of the Lemma]
Let us fix $L\in \N$. 
We set 
 $\hat \pi^i|_{\{ \ell_i \geq \frac 1 L\}}= \hat \pi^{L,i}$, $i\in \N$. 
 
After extracting subsequence $\hat \pi^{L, i}$ converges weakly to a measure $\hat \pi^{L,\sinfty}$ concentrated in $\{ \ell_\sinfty\geq \frac 1 L\}=: \ol V^L_\sinfty$. 
\smallskip\\
{\it Claim:}  $\hat \pi^{L, i}\leq  \hat \pi^\sinfty|_{\ol V^L_\sinfty}$. 
\smallskip\\
{\it Proof of the Claim:}
Let us pick an arbitrary nonnegative, continuous functions $\phi$ on $Z\times Z$. By weak convergence we have 
$$ \int \phi  \de \hat \pi^{L, \sinfty}=\lim_{i\rightarrow \infty} \int \phi \de \hat \pi^{L, i} = \lim_{i\rightarrow \infty} \int_{\ol V_i^{L}} \phi \de \hat \pi^i\leq \int_{\ol V_\sinfty^{L}} \phi \de \hat \pi^\sinfty.$$
This is the claim.
\hfill $\triangle$
\smallskip

As in the previous lemma we can prove that $\lim_{i\rightarrow \infty} \int \ell_i^p\de \hat \pi^{L, i} = \int \ell_\sinfty^p \de \hat \pi^{L, \sinfty}$. Hence
$$\lim_{i\rightarrow \infty} \int \ell_i^p\de \hat \pi^{L, i} = \int \ell_\sinfty^p \de \hat \pi^{L, \sinfty} \leq \int \ell_\sinfty^p \de \hat \pi_\sinfty|_{\ol V^L_\sinfty}\leq \int \ell_\sinfty^p \de \hat \pi^\sinfty.$$
Therefore, given $\eta>0$ there exists $i(\eta)$ such that for $i\geq i(\eta)$ we have 
$$\int\ell_i^p \de \hat \pi^{L, i}\leq \int \ell^p_\sinfty \de \hat \pi^\sinfty+ \eta$$ 
and by monotone convergence w.r.t. $L\rightarrow \infty$ we have 
$$\int_{\{\ell_i>0\}}\ell_i^p \de \hat \pi^i\leq \int \ell^p_\sinfty \de \hat \pi^\sinfty+ \eta$$ 
This yields
$$\limsup_{i\rightarrow \infty} \int \ell_i^p \de \hat \pi^i =\limsup_{i\rightarrow \infty} \int_{\{\ell_i>0\}} \ell_i^p \de \hat \pi^i\leq \int \ell_\sinfty^p \de \hat \pi^\sinfty.$$
On the other hand, by \eqref{ell_conv_2} we have
$$\int \ell_i^p \de \hat \pi^i \geq \int \ell_i^p \de (\pi')^{i,i,i}\rightarrow \int \ell_\sinfty^p \de \pi^\sinfty.$$
Consequently 
$\int \ell_\sinfty^p \de \pi^\sinfty\leq \int \ell_\sinfty^p \de \hat \pi^\sinfty.$
Since $\pi^\sinfty$ is  optimal  between $\mu_0$ and $\mu_1$, also $\hat \pi^\sinfty$ is an optimal coupling. 
\end{proof}
\noindent
The proof of the previous lemma shows in particular that 
\begin{align}\label{ell_conv_3}\lim_{i\rightarrow \infty} \int \ell_i^p \de \hat \pi^i= \int \ell_\sinfty^p \de \hat \pi^\sinfty. \end{align}

\noindent
{\bf \large (8)}
 Since $\hat \pi^\sinfty$ is optimal and since $\mu_0$ and $\mu_1$ are strongly $p$-dualizable, it follows that $\hat \pi^\sinfty$ is concentrated on $\Gamma$. Hence 
$$\hat \pi^\sinfty(\{\ell_\sinfty>0\})=1.$$
It follows that $\hat \pi^\sinfty(\{ \ell_\sinfty\geq \lambda\})\rightarrow 1$ for $\lambda\rightarrow 0$.

Hence, similar as in the previous step {\bf (5)} we can show the following.  
If we choose $\lambda_i\downarrow 0$ and pick the diagonal sequence, it follows that 
$$\lim_{i\rightarrow \infty} \hat \pi^i(\ell_i\geq \lambda_i)= 1.$$

Hence, we define 
$$c^i:= \hat \pi^i(\ell_i\geq \lambda_i), \ \  (c^i)^{-1} \hat \pi^i|_{\{\ell_i\geq \lambda_i\}}= (\hat \pi')^{i}.$$
It follows $ (\hat \pi')^i\rightarrow \hat \pi^\sinfty \mbox{ weakly}.$

We also set $(\mu_{j}')^i:= (P_j)_\sharp (\hat \pi')^{i}$. Since the restriction of an optimal coupling still is optimal \cite[Lemma 2.10]{camolorentz}, we have $(\hat \pi')^i\in \Pi_{\leq}^{p-opt}(\mu_{0}^i, \mu_{1}^i).$ Moreover, $(\hat \pi')^{i}(\{\ell_i>0\})=1$. Hence $(\mu_{0}^i, \mu_{1}^i)$ are $p$-dualizable.
\smallskip\\
{\bf (8.1)}
Repeating the arguments from before we obtain that 
\begin{align}\label{ei5}\limsup_{i\rightarrow \infty} S_N((\mu_{j}')^i|\m_i)\leq \limsup_{i\rightarrow \infty} S_N(\hat \mu_j^i|\m_i).\end{align}
\\
{\bf (9)}
Combining \eqref{ei1}, \eqref{ei2}, \eqref{ei3}, \eqref{ei4} and \eqref{ei5} it follows
$$\limsup_{i\rightarrow \infty} S_N((\mu'_j)^i|\m_i)\leq S_N(\mu_{j}), \ \ j=0, 1.$$
The probability measures $(\mu_0')^i=: \mu_0^i$ and $(\mu_1')^i=:\mu_1^i$ are $\m_i$-absolutely continuous,  have bounded density, and are timelike $p$-dualizable. Since $Y_i$ satisfies the condition $\TCD_p(0,N)$, 
$\forall i \in \N$ there exists an $\ell_p$-geodesic $(\mu_t^i)_{t\in [0,1]}$ such that 
$$S_N(\mu^i_t|\m_i) \leq (1-t) S_N(\mu^i_0|\m_i) + t S_N(\mu^i_1|\m_i).$$
Let $\eta^i$ be an $\ell_p$-optimal dynamical coupling representing the $\ell_p$-geodesic $(\mu_t^i)_{t\in [0,1]}$, i.e. $(e_t)_\sharp \eta^i= \mu_t^i$ for all $t\in [0,1]$. 

\smallskip
Since $Y_i$ is weak causal convex,  $\eta_i$ is concentrated on the set 
$$\Gamma_i^k= \left\{ \gamma \in \Geo(Y_i): \gamma(0), \gamma(1)\in \ol U^k_i \mbox{ and } \Im \gamma\subset \ol U^{k+1}_i\right\}$$
and 
$$\bigcup_{i\in \N\cup\{\infty\}} \Gamma_i^k\cup \Gamma_\sinfty^k \mbox{ is compact w.r.t. } \de^\sinfty.$$

By Prokhorov's Theorem we can therefore extract a subsequence from $(\eta^i)_{i\in \N}$ that converges weakly to a probability measure that $\eta^\sinfty$ that is supported in $\Gamma_\sinfty^k$. 

It holds that $(e_0, e_1)_\sharp \eta^\infty= \tilde \pi$ is a coupling between $\mu^\sinfty_0$ and $\mu^\sinfty_1$, and 
\begin{enumerate}
\item $(e_0, e_1)_\sharp \eta^i\rightarrow \tilde \pi$ weakly for $i\rightarrow \infty$, 
\item $(e_t)_\sharp\eta^i= \mu_t^i\rightarrow (e_t)_\sharp \eta^\sinfty= \mu_t^\sinfty$ weakly for $i\rightarrow \infty$. 
\end{enumerate}
In particular, by joint lower semi-continuity of the $N$-Renyi entropy we have that 
$$\liminf_{i\rightarrow \infty} S_N(\mu_t^i|\m_i)\geq S_N(\mu^\sinfty_t|\m_\sinfty).$$
It follows that 

$$S_N(\mu^\sinfty_t|\m_\sinfty) \leq (1-t) S_N(\mu^\sinfty_0|\m_\sinfty) + t S_N(\mu^\sinfty_1|\m_\sinfty).$$

Finally, we have to show that $\tilde \pi$ is an optimal coupling. Then, it follows that $t\in [0, 1]\mapsto (e_t)_\sharp \eta^\sinfty= \mu_t^\sinfty$ is an $\ell_p$-geodesic between $\mu_0^\sinfty$ and $\mu_1^\sinfty$.  This can be done along the same lines as in the proof of Lemma \ref{lem:optimal_coupling}. 

This shows the condition $\TCD_p(0,N)$. 
\medskip\\
{\bf (10)} The proof for the condition $\TCD_p^e(0,N)$ is very similar. One has to replace the $N$-Renyi entropy with the function $U_N$ that is a composition of the Boltzmann-Shanon entropy and $e^{\frac{1} N x}$. $U_N$ has the same semi-continuity properties under weak convergence. Moreover, under the previous truncation arguments for the involved  couplings $U_N$ also satisfies the inequalities that we derive for $S_N$ in {\bf (3.1)}, {\bf (4.1)},  {\bf (5.1)} and {\bf (8.1)}. This is again similar to  the proof of Theorem 3.15 in \cite{camolorentz}.

The proof in the case of the condition $\TCD_p^e(K,N)$ then only requires as an additional ingredient the stability of the $L^2((\hat \pi')^i)$-norms of the signed time separation function $\ell_i$ along the final sequence of couplings $(\hat \pi')^i$. 
For this recall  first \eqref{ell_conv_3} for the couplings $\hat \pi^i$ converging to $\hat \pi^\sinfty$. Now $(\hat \pi')^i$ is obtained by restriction of $\hat \pi^i$ to $\{\ell_i\geq \lambda_i\}$ for $\lambda_i\downarrow 0$. 
For every $i\in \N$ we choose $\lambda_i>0$ small enough such that 
$$\int \ell_i^2 \de(\hat \pi')^i- \varepsilon \leq \int \ell_i^2\de \hat \pi^i \leq \int \ell_i^2\de (\hat \pi')^i+ \varepsilon. $$
Hence 
$$ \left\| \tau_i \right\|^2_{L^2((\hat \pi')^i)}= \int \ell_i^2 \de(\hat \pi')^i \rightarrow \int \ell_\sinfty^2 \de \hat \pi^\sinfty= \left\| \tau_\sinfty\right\|^2_{L^2(\hat \pi^i)}.$$
This follows from  equation \eqref{ell_conv_2} for $\hat \pi^i$ and the definition of $(\hat \pi')^i$.
This finishes the proof.
\end{proof}
\begin{theorem}\label{th:stabilityTMCP}
Let $(Y_i, \de_i, \ell_i, \mathcal U_i, \m_i)_{i\in \N}$ be a sequence of properly covered, measured Lorentzian geodesic spaces that satisfy  $\TMCP^e(K,N)$. We assume $(Y_i, \de_i, \ell_i, \mathcal U_i, \m_i)_{i\in \N}$ converges in the measured  $\ell$ sense to a properly covered, measured Lorentzian geodesic space $(Y, \de, \ell, \mathcal U, \m)$.  Let $\mathcal U_i$ be weak causal convex $\forall i\in \N$. 

Then $(Y, \de, \ell, \mathcal U, \m)$ satisfies the property $\TMCP^e(K,N)$. 

If $Y_i$ satisfies the property $\TMCP^e(K,N)$ locally in $U^k_i$ for $k\in \N$ and for all $i\in \N$, then $Y$ satisfies the  property $\TMCP^e(K,N)$ locally in $U^k$. 
\end{theorem}
\begin{proof} We fix $\mu_0^\sinfty= \rho_0^\sinfty \de \m_\sinfty$ in $\mathcal P_c(Y_\sinfty) \cap D(\Ent_{\m_\sinfty})$ and $x_1^\sinfty$ such that $x\ll x_1^\sinfty$ for $\mu_0^\sinfty$-a.e. $x\in X_\sinfty$. We can assume that  $\mu_0^\sinfty (\ol U^k_\sinfty)=1$,  $x_1^\sinfty \in \ol U_\sinfty^k$ and $J^+(\supp \mu_0^\sinfty)\cap J^-(x_1^\sinfty)\subset \ol U^{k+1}_\sinfty$.  We fix a metric space $Z$ where $\ol U^{k+1}_i$ and $\ol U^k_i$ converge in Hausdorff sense to $\ol U^{k+1}_\sinfty$ and $\ol U^{k}_\sinfty$, respectively.  In the following we omit the superscript $k$.  We also set $\bar \m_i= \m_\sinfty(\ol U_i)^{-1} \m_i|_{\ol U_i}$ for $i\in \N\cup {\infty}$.  
\medskip

We  assume  that $\rho_0^\sinfty\in L^\sinfty(\bar \m_\sinfty)$ and $\mu_0^\sinfty(\{\ell_\sinfty(\cdot, x_1^\sinfty)\geq \frac 1 l\})=1$ for $l\in \N$. We will remove this assumption at the end of the proof ($\dagger$).

Let $\boldsymbol p^i$ be the  coupling between the measures $\bar \m_\sinfty$ and $\bar \m_i$ that we introduced in the previous proof. Using the coupling $\boldsymbol p^i$ we construct a sequence of probability measures $\hat \mu_0^i = \hat \rho_0^i \de \m^i$ such that 
\begin{enumerate}
\item $\left\| \hat \rho^i_0\right\|_{L^\sinfty} \leq \left\| \rho_0^i \right\|_{L^\sinfty}.$
\item $\Ent_{\m^i}(\mu_0^i) \leq \Ent_{\m^\sinfty}(\mu_0^\sinfty)$.
\item $\mu_0^i \rightarrow \mu_0^\sinfty$ weakly in $Z$. 
\end{enumerate}
This is very similar to the construction in {\bf (4)} in the proof of the previous theorem (also compare with {\bf Step 1a} in the proof of Theorem 3.14 in \cite{camolorentz}).

\medskip

From Hausdorff convergence of $\ol U_i$ ot $\ol U_\sinfty$ we find a sequence $x^i_1\in \ol U_i$ that converges to $x_1$ in $Z$.
From uniform convergence of $\ell_i$ to $\ell_\sinfty$ we have  that there is some $i_0$ s.t. \begin{center}$B^Z_\delta(\{\ell_\sinfty(\cdot, x^\sinfty_1)\})\cap \ol U_i\subset \{\ell_i(\cdot, x^i_1)\geq \frac 1 {2 l}\}$ for $i\geq i_0=i(l,k)$
\end{center}
(Remark \ref{rem:uniform_convergence}).
Hence, by weak convergence, $$\liminf_{i\rightarrow \infty} \mu^i_0 (\{\ell_i(\cdot, x^i_1)\geq \frac 1 {2 l}\})\geq \mu^\sinfty_0( \{\ell_\sinfty(\cdot, x^\sinfty_1)\}\geq \frac 1 l )=1.$$
Therefore, we define $(\mu^l)_0^i= (b^{l,i})^{-1} \mu_0^i|_{\{\ell_i(\cdot, x_1^i)\geq \frac 1 {2l}\}}.$ It follows
\begin{enumerate}
\item $(\mu^l)_0^i\rightarrow \mu^\sinfty_0$ weakly, 
\item $\limsup \Ent_{\m_i}((\mu^l_0)^\sinfty)\leq \Ent_{\m_\sinfty}(\mu_0^\sinfty).$
\end{enumerate}
Since along the sequence we have the condition $\TMCP^e(K,N)$, it follows
$$U_N(\hat \mu_t^i)\leq \sigma_{K,N}^{(1-t)}\left(\left\|\tau(\cdot, x_1^i)\right\|_{L^2((\mu^l)_0^i)} \right)U_N(\hat \mu_0^i).$$
for an $\ell_p$-geodesic $\hat \mu_t^i$ between $(\mu^l)_0^i$ and $\delta_{x_1^i}$.
\bigskip

Finally, we have to remove the assumptions ($\dagger$).
To achieve this we define 
$$(\mu_0^l)^\sinfty=\frac 1 {c_l} \mu_0^\sinfty|_{\{ \ell_\sinfty(\cdot, x_1^\sinfty)\geq \frac 1 l\}}$$
where $c_l= \mu_0^\sinfty(\{ \ell_\sinfty(\cdot, x_1^\sinfty)\})\rightarrow 1$ as $l \rightarrow \infty$. It follows that for $l\rightarrow \infty$
\begin{enumerate} 
\item $(\mu_0^l)^\sinfty\rightarrow \mu_0^\sinfty$ weakly,
\item $\limsup \Ent_{\m_\sinfty}((\mu_0^l)^\sinfty)\leq \Ent_{\m_\sinfty}(\mu_0^\sinfty).$
\end{enumerate}

Moreover, we define $\bar (\rho^l)^\sinfty_{0,k}= \frac 1 {\bar c_{0,k}} \min\{ (\rho^l)_0^\sinfty, k\}$ where $\bar c_{0,k}$ is chosen such that $(\mu^l)_{0,k}^\sinfty= (\rho^l)^\sinfty_{0,k} \de \m_\sinfty$ is a probability measure.  Again, it follows for $k\rightarrow \infty$
\begin{enumerate} 
\item $(\bar \mu^l)_{0,k}^\sinfty\rightarrow \mu_0^\sinfty$ weakly,
\item $\lim\Ent_{\m_\sinfty}((\mu^l)_{0, k}^\sinfty)= \Ent_{\m_\sinfty}((\mu^l)_0^\sinfty).$
\end{enumerate}
By stability of the defining inequality for the measure contraction property, we can conclude.
\end{proof}
\subsection{Stability synthetic timelike sectional curvature bounds}
\begin{theorem}\label{th:stabilityTCBB}
Let $(Y_i, \de_i, \ell_i, \mathcal U_i)_{i\in \N}$ be  properly covered Lorentzian geodesic spaces that have globally timelike sectional curvature bounded from below by $K$, and  converge to a properly covered Lorentzian geodesic space $(Y, \de, \ell, \mathcal U)$.  Then $(Y, \de, \ell, \mathcal U)$ has  timelike sectional curvature bounded from below by $K$. 
\end{theorem}
\begin{jjj}
It will be become clear from the proof that the following stronger statement holds: {\it 
If $U^k_i$ is $\geq K$-comparison neighborhood for $k\in \N$ and for all $i\in \N$, then $U^k$ is a $\geq K$-comparison neighborhood. }
\end{jjj}
\begin{proof}
Since $Y_i$ has globally timelike sectional curvature bounded from below by $K$, its time separation function $\tau_i$ is continous. If follows from Lemma \ref{lem:cont_tau} that the time separation function $\tau$ of $Y$ is continuous as well. 

Let $(y, x, z_1, z_2)$ be  timelike future  endpoint-causal 4-point configuration  in $Y$ such that 
$\tau(y,z_2)< \pi_{-K}$ and let $(\bar y, \bar x, \bar z_1, \bar z_2)$  be  a comparison configuration in $\mathbb L^2(K)$. 

For $z_2$ we can choose, $z_2'\in I(z_2)$ such that $z_2'$ is arbitrarly close to $z_2$ and $z_2 \ll  z_2'$. It follows from Lemma 2.10 in \cite{kusa} that $z_1\ll  z_2'$. Hence $y\ll x \ll z_1\ll z_2'$.

By abuse of notation we will write $z_2$ instead of $z_2'$ in the following. 
There exists $\el\in \N$ such that  
 $$\min \left\{ \tau(x, y), \tau(y, z_i), \tau(x, z_i), \tau(z_1, z_2), : i\in \{1, 2\}\right\}\geq \frac 1 \el.$$
 It follows that there is a sequence of points $y^i, x^i, z_1^i, z_2^i\in Y_i$ that converge in $Z$ to $y, x, z_1, z_2$ such that 
 \begin{align*}
 \min\left\{ \tau(x^i, y^i), \tau(x^i, z^i_j), \tau(y^i, z^i_j), \tau(z_1, z_2), j\in \{1,2\}\right\}\geq \frac 1 {2 \el}.
 \end{align*}
 Hence $y^i\ll x^i \ll z_1^i\ll z_2^i$. Hence, $(y^i, x^i, z_1^i, z_2^i)$ is a timelike future endpoint-causal 4-point configuration. Moreover, the values of the time separations functions converge to the values in the limit. 
 
 We pick a comparison configuration $(\bar x^i, \bar y^i, \bar z_1^i, \bar z_2^i)$. It follows that 
 $$ \bar \tau(\bar z^i_1, \bar z^i_2)\leq \tau(z^i_1, z^i_2) = \ell(z_1^i, z_2^i).$$ The values of the time separation function $\bar \tau$ on this configuration in the model space converge to the values of $\bar \tau$ on $(\bar x, \bar y, \bar z_1, \bar z_2)$. 
It follows that 
 \begin{align*} \bar \tau(\bar z_1,\bar z_2)= \lim_{i\rightarrow \infty}\bar \tau(\bar z^i_1, \bar z^i_2)\leq \limsup_{i\rightarrow \infty} \ell(z^i_1, z^i_2)\leq \ell(z_1, z_2)= \tau(z_1, z_2).\end{align*}

We return to the notation $z'_2$ for $z_2$. We can send $z'_2$ to $z_2$. Then, by continuity of $\tau$ the values of the time separation function on $(y, x, z_1, z_2')$ converges to the values of the time separation function on $(y, x, z_1, z_2)$. Similarly, the values of the time separation function of the comparison configurations converge.  It follows that $\bar\tau(\bar z_1, \bar z_2) \leq \tau(z_1, z_2)$. 
 \end{proof}
 {\color{black}
 \begin{remark}
 For regular, connected, globally hyperbolic Lorentzian length spaces that have a timefunction timelike sectional lower curvature bounds globalize \cite{bhnr} and we can improve our theorem in this direction. 
 \end{remark}
 }
\section{$\ell$-convergence of generalized cones}\label{sec:concon}
\begin{lemma} Let $I$ be a compact interval.
We consider 2 continuous  functions $f, g: I\rightarrow [0, \infty)$
such that 
$g\leq f$. Let $L>0.$
Then we have that 
$$\{\ell_{-I\times_f [0,L]}\geq 0\}\subset \{ \ell_{-I\times_g [0,L]}\geq 0 \}$$
and for all $(s,x), (t,y) \in I\times [0,L]$ it holds
$$\ell_{-I\times_f [0,L]}(s,x,t,y)\leq \ell_{-I\times_g [0,L]}(s,x,t,y) .$$
\end{lemma}
\begin{proof}
We have that $$\left((s,x), (t,y)\right)\in \left\{\ell_{-I\times_f [0,L]}\geq 0\right\}$$ if and only if there exists a future directed causal curve $\gamma=(\alpha, \beta): [a,b] \rightarrow I\times [0,L]$ w.r.t. $-(\de t)^2 + f^2 (\de r)^2$ from $(s,x)$ to $(t,y)$. More precisely,  $\gamma$ is an admissible path, $\alpha'\geq 0$ and 
$$(\alpha')^2- (f\circ \alpha)^2 (\beta')^2\geq  0 \ \ \mathcal L^1\mbox{-a.e. in } [a,b].$$
Since $g\leq f$, it follows that
$$(\alpha')^2- (g\circ \alpha)^2 (\beta')^2\geq  (\alpha')^2- (f\circ \alpha)^2 (\beta')^2\geq  0.$$
It follows that
 $\gamma$ is also a future directed causal curve w.r.t. $-(\de t)^2 + g^2 (\de r)^2$. Hence $\left((s,x), (t,y)\right)\in  \{ \ell_{-I\times_g [0,L]}\geq 0 \}$. By definition of the time separation function,  we also have that
$$\tau_{-I\times_f [0,L]}((s,x),(t,y))\leq \tau_{-I\times_g [0,L]}((s,x),(t,y)).$$ 

If $((s,x), (t,y) )\notin \{\ell_{-I\times_f [0,L]}\geq 0\}$, then $\ell_{-I\times_f [0,L]} (s,x,t,y)=-\infty$ and the inequality holds as well.
\end{proof}

Let $I_i$, ${i\in \N\cup\{\infty\}}$, be  a family of compact intervals such that $I_i\overset{GH}{\rightarrow}I_\sinfty$, and 
let $f_i: I_i\rightarrow (0, \infty)$ be  a sequence of continuous functions that converge  uniformly to a  function $f_\sinfty: I_\sinfty\rightarrow (0, \infty)$.  
We assume that $f_i=  f_i|_{I_i}$ for   continuous functions $ f_i: J_i \rightarrow (0, \infty)$ and  intervals $J_i$ such that  $I_i\subset J_i$, and $J_i$ is open and relatively compact. We also assume $J_i \overset{\scriptscriptstyle GH}{\rightarrow} J_\sinfty$.
Then,  for all $ i \in \N\cup \{\infty\},$ we can define the generalized cone $-J_i\times_{f_i} [0, L]$. We consider $I_i\times [0,L]$ as a subset of   $-J_i\times_{f_i} [0,L]$.
Convergence of $J_i$ to $J_\sinfty$ implies GH convergence of $J_i\times [0,L]$ to $J_\sinfty\times [0,L]$.  We consider $J_i\times [0,L]$ equipped with the standard Euclidean  metric $\de_{J_i\times [0,L]}$.
 For every $i\in \N$ we find an isomorphism $\psi_i$ between $J_i$ and $J_\sinfty$. 
Then, a $\delta$-GH-isometry is given by $(s,x) \mapsto (\psi(s),x)$.
\begin{proposition}\label{prop:uniform_convergence_cone} Let $f_i$ be as before. We set $\de_{J_i\times [0,L]}=\de_i$, $i\in \N\cup \{\infty\}$.
For all $\el \in \N$ and $\epsilon \in (0, \frac 1 {2\el})$, there exists $\delta(\epsilon, \el, f_\sinfty, L)=\delta_0>0$ and $i(\epsilon, \el, f_\sinfty, L)=i_0\in \N$ such that for all $i\geq i_0$ and for all $\delta\in (0, \delta_0)$ the following holds: 
\begin{enumerate}
\item[(1)] if $(s,x), (t,y) \in \{ \ell_{-J_i\times_{f_i} [0,L]}\geq 0\}\cap \left( I_i\times [0,L]\right)$, then 
$$\ell_{-J_i\times_{f_i} [0,L]}(s,x, t,y)\leq \ell_{- J_\sinfty \times_{f_\sinfty} [0,L]}(\tilde s,\tilde x,\tilde t,\tilde y)+ \epsilon$$ 
\end{enumerate} 
for all
$(\tilde s, \tilde x), (\tilde t, \tilde y)\in I_\sinfty\!\times [0,L]$ s.t. \!$\de_{\sinfty}(\tilde s,\tilde x, \psi_i(s, x)){\scriptstyle+}\de_{\sinfty}(\tilde t,\tilde  y,\psi_i(t, y))\leq\delta.$
\begin{enumerate}
\item[(2)] if $(s,x), (t,y) \in \{ \ell_{-J_\sinfty\times_{f_\sinfty}[0,L]}\geq \frac 1 \el\}\cap \left( I_\sinfty \times [0,L]\right)$, then 
$$ \ell_{-J_i\times_{f_i} [0,L]}(s,x, t,y) \geq \ell_{-J_\sinfty\times_{f_\sinfty} [0,L]}( s, x, t, y)- \epsilon$$
\end{enumerate}
for all $(\tilde s, \tilde x), (\tilde t, \tilde y)\in I_i\times [0,L]$ s.t.
$\de_{\sinfty}(s,x, \psi(\tilde s, \tilde x))+\de_{\sinfty}(t, y,\psi( \tilde t, \tilde y))\leq \delta.$
%
\end{proposition}

\begin{proof} {\bf (0)} Since for every $i\in \N$ we find an isomorphism $\varphi=\psi^{-1}$ between $J_\sinfty$ and $J_i$, we can define $\varphi^* f_i= \tilde f_i$. Then, the generalized cone $-J_i\times_{f_i}[0,L]$ is isomorphic to $-J_\sinfty\times_{ f_i} [0,L]$ via $\phi(s,x)= (\varphi(s), x)$ (Remark \ref{rem:trans_cone}). Therefore, w.l.o.g. we can assume  $J_i= J_\sinfty=: J$ and $I_i=I_\sinfty=:I$. \smallskip\\
{\bf (1)} We prove {\it (2)}.
Let  $\eta>0$. There exists $i(\eta)\in \N$ s.t. $\forall i\geq i(\eta)$ we have 
$$f_\sinfty(t)-\eta \leq f_i(t) \leq f_\sinfty(t) + \eta  \ \ \forall t\in I.$$
It follows from the previous lemma that
$$ \ell_{-J\times_{(f_\sinfty+\eta)}[0,L]}(s,x,t,y) \leq \ell_{-J\times_{f_i}[0,L]}(s,x, t,y)\leq \ell_{-J\times_{(f_\sinfty-\eta)}[0,L]}(s,x,t,y) \ \ $$
for all $(s,x), (t,y) \in I\times [0,L]$. 
\smallskip

Let $(s,x), (t,y)\in \{ \ell_{-J\times_{f_\sinfty} [0,L]}\geq  \el^{-1}\}\cap (I\times [0,L])$
and let $\epsilon \in (0, \frac 1 {2\el})$.
Since the time separation function $\tau_{-J\times_{f_\sinfty}[0,L]}$ for the generalized cone $-J\times_{f_\sinfty} [0,L]$  is continuous  (Theorem \ref{Pr: various properties}), there exists $\delta(\epsilon, f_\sinfty,  L)=\delta_0>0$ such that for every  $\delta\in (0, \delta_0)$ we have 
$$\left|\ell _{-J\times_{f_\sinfty} [0,L]}( s,  x,  t,  y)-\ell _{-J\times_{f_\sinfty} [0,L]}(\tilde s, \tilde x, \tilde t, \tilde y)\right|<\epsilon/2$$ 
$\forall (\tilde s, \tilde x), (\tilde t, \tilde y)\in J\times [0, L]$ with $\de_{\sinfty}(s, x, \tilde s, \tilde x)+ \de_{\sinfty}( t, y, \tilde t, \tilde y)\leq \delta$. 

We choose  two such points $(\tilde s, \tilde x)$ and $(\tilde t, \tilde y)$ in $I\times [0,L]$.
There exists an admissible, future oriented timelike  curve $\gamma=(\alpha, \beta):[a,b]\rightarrow I\times [0,L]$, that is also a geodesic, connecting $(\tilde s,\tilde  x)$ and $(\tilde t,\tilde y)$ such that $L^{-J\times_{f_\sinfty} [0,L]}(\gamma)\geq \frac 1 {2\el}$.

We  assume that $\gamma$ is parametrized according to $\left(\tau_{-J\times_{f_\sinfty} [0,L]}\right)$-arclength. Then it follows
$${  (\alpha')^2 - (f_\sinfty^2 \circ \alpha) (\beta')^2}\geq \frac 1 {4\el^2}.$$
For $\eta>0$ that is sufficiently small,  it follows 
$${  (\alpha')^2 - \left((f_\sinfty+\eta)^2 \circ \alpha\right) (\beta')^2}>0.$$
Hence, $\gamma$ is a causal curve in $-J\times_{(f_\sinfty+\eta)} [0,L]$.

It also follows that
$$\sqrt{ (\alpha')^2 - (f_\sinfty+\eta)^2 \circ \alpha (\beta')^2} \geq \sqrt{(\alpha')^2 - f_\sinfty^2 \circ \alpha (\beta')^2} - C(f_\sinfty)\sqrt{ \eta} |\beta'| $$
for a constant $C(f_\sinfty)>0$. 
We obtain that 
$$\ell_{-J\times_{(f_\sinfty+\eta)} [0,L]}(\tilde s,\tilde x,\tilde t,\tilde y)\geq \ell_{-J\times_{f_\sinfty} [0,L]}(\tilde s,\tilde x,\tilde t,\tilde y)- C(f_\sinfty)  \sqrt \eta L.$$
If we choose $\eta>0$ smaller than ${\epsilon^2}/{(4 C(f_\sinfty)^2 L^2)}$, and $i\geq i(\eta)= i(\eta(\epsilon, f_\sinfty,L))=:i(\epsilon, f_\sinfty, L)$, it follows that
$$ 
 \ell_{-J\times_{f_i}[0,L]} (\tilde s,\tilde x,\tilde t,\tilde y) \geq\ell_{-J\times_{f_\sinfty} [0,L]}( s, x, t, y) -\epsilon. $$
{\bf (2)}  We prove {\it (1)}.
We argue by contradiction. Assume there exists $\epsilon'>0$ such that for every $i\in \N$, there exist   $(\delta_i)_{i\in \N}$ with $\delta_i\downarrow 0$  for $i\rightarrow \infty$ and  points $(s_i, x_i), (t_i,y_i), (\tilde s_i, \tilde x_i), (\tilde t_i, \tilde y_i) \in I\times [0,L]$ with $\ell_{-J\times_{f_i} [0, L]}(s_i,x_i, t_i,y_i)\geq 0$ as well as $\de_{\sinfty}(s_i, x_i, \tilde s_i, \tilde x_i)+ \de_{\sinfty}(t_i, y_i, \tilde t_i, \tilde y_i)\leq \delta_i$ but 
$$\ell_{-J\times_{f_i} I}(s_i,x_i, t_i,y_i)> \ell_{-J\times_{f_\sinfty} I}(\tilde s_i,\tilde x_i,\tilde t_i,\tilde y_i)+ \epsilon'.$$

There exists a maximal causal curve $\gamma_i$ in $ I\times [0,L]\subset-J\times_{f_i}[0,L]$ connecting $(s_i,x_i)$ and $(t_i,y_i)$. Since $J$ is relatively compact, the family $\{ -J\times_{f_i}[0,L]\}_{i\in \N}$ is  uniformly non-totally imprisoning. Hence, the metric speed w.r.t.  $\de_i$ of $\gamma_i$ is uniformly bounded by a constant $C>0$. By the Arzela-Ascoli theorem, we can extract  a  subsequence that converges uniformly to a curve $\gamma$ in $I\times [0,L]$ connecting the points $(\hat s,\hat x)$ and $(\hat t,\hat y)$ such that $L^{J\times_{f_\sinfty} [0,L]}(\gamma)\leq  C$.  In particular $(\tilde s_i, \tilde x_i)\rightarrow (\hat s, \hat x)$ and $(\tilde t_i, \tilde y_i)\rightarrow (\hat t, \hat y)$. 
For $\eta>0$ we find $i(\eta)\in \N$ such that for every $i\geq i(\eta)$ we have $f_\sinfty-\eta \leq  f_i$.  From this it is easy to see that  $\gamma_i$ is a causal curve in $-J\times_{(f_\sinfty-\eta)} [0,L]$. 

Then, we  apply the limit curve Theorem \ref{th:lim_cur} that yields that $\gamma$ is causal in $-J\times_{(f_\sinfty-\eta)} [0, L]$ as well. Since $\eta>0$ is arbitrary $\gamma$ is a causal in $-J\times_{f_\sinfty} [0,L]$.

 Moreover, by Proposition \ref{prop:limsup} we have
\begin{align*}&\limsup \ell_{-J\times_{f_\sinfty}[0,L]}(\tilde s_i, \tilde x_i, \tilde t_i, \tilde y_i)+\epsilon'\leq \limsup \ell_{-J\times_{f_i}[0,L]}( s_i,  x_i,  t_i,  y_i)\\
&= \limsup L^{-J\times_{f_i}[0,L]}(\gamma_i)\leq \limsup L^{-J\times_{(f_\sinfty-\eta)}[0,L]}(\gamma_i) \leq L^{-J\times_{(f_\sinfty-\eta)}[0,L]}(\gamma).\end{align*}
Consequently, by lower semi-continuity of $\ell_{-J\times_{f_\sinfty}[0,L]}$ and since $\eta>0$ was arbitrary, it follows
$$\ell_{-J\times_{f_\sinfty} [0,L]}(\hat s,\hat x,\hat t,\hat y)+\epsilon'
\leq  L^{-J\times_{f_\sinfty}[0,L]}(\gamma)
\leq \ell_{-J\times_{f_\sinfty}[0,L]}(\hat s, \hat x, \hat t, \hat y).
$$
This is the contradiction.
\end{proof}
\begin{definition}[Properly covered generalized cones]\label{def:covercone}
Let $I$ be an open interval and let $I^k\subset I$, $k\in \N$, be  open, relatively compact  intervals such that
\begin{enumerate}
\item $\diam{I^k}\leq \min\{ k, \diam I\}$ for all $k\in \N$, 
\item  $I^{k}\subset I^{k+1}$ for all $k\in \N$,
\item  ${\bigcup_{k\in \N} I^k } =   I$.
\end{enumerate}

Let $f: I \rightarrow [0, \infty)$ be a continuous function,  and let $(X, o)$ be a geodesic metric space.

We consider the generalized cone $-I\times_f X=Y$. A proper cover of $Y$ is given by $I^k\times_f B_{2^k}(o)= U^k$, $k\in \N$. We set $Y^k= I^k\times \bar B_{2^k}(o)$. 

We call $({}^-I\times_f X, (U^k)_{k\in \N})=Y$ a properly covered generalized cone. In particular, $Y$ is a properly covered Lorentzian geodesic space.

In the following, we will denote a properly covered generalized cone also just with $-I\times_f X$ and then assume the existence of $(Y^k)_{k\in \N}$ implicitly. 
\end{definition}
\begin{theorem}\label{th:convergencecones}
Let $Y_i:= ({}^-I_i \times_{f_i} X_i, (Y_i^k)_{k\in \N})$, $i \in \N\cup \{\infty\}$,  be  properly covered generalized cones. We assume that 
\begin{enumerate}
\item $(X_i, o_i)\rightarrow (X_\sinfty, o_\sinfty)$ in pointed GH sense, 
\item $I^k_i\rightarrow I^k_\sinfty$ in GH sense for all $k\in \N$, 
\item $f_i|_{I_i^k}$ converges uniformly to $f_\sinfty|_{I^k_\sinfty}$ as $i\rightarrow \infty$ for all $k\in \N$. 
\end{enumerate}

Then, the sequence of  properly covered Lorentzian length spaces  $Y_i$ converges in $\ell$-sense  to the properly covered Lorenztian length space $Y_\sinfty$. 

If every $Y_i$ is equipped with measures $f^N_i \de t\otimes \de \m_{X_i}$, i.e. we consider properly covered generalized $N$-cones $-I_i\times_{f_i}^N X_i=Y_i$, then $Y_i$ converges in the $\m\ell$-sense to $Y_\sinfty$. 
\end{theorem}
\begin{definition}
Given a sequence $f_i: I_i\rightarrow [0, \infty)$, $i\in \N\cup\{\infty\}$ and coverings $\{I^k_i\}$ for $I_i$, $i\in \N\cup\{\infty\}$ such that the assumptions (2) and (3) in the previous theorem hold, then we say that {\it $f_i$ converges uniformly to $f_\sinfty$ subject to $\{I^k_i\}$.}
\end{definition}
\begin{proof}
We se ${\bar B_{2^k}(o_i)}=X^k_i\subset X_i$, $j\in \N\cup\{\infty\}$. From pointed GH convergence we know that $X^k_i \rightarrow X_\sinfty^k$ in GH sense $\forall k\in \N$. Since also $I^k_i \overset{GH}{\longrightarrow} I^k_\sinfty$ for all $k\in \N$, it follows that $I^k_i \times X^k_i = Y^k_i\overset{GH}{\longrightarrow} I^k_\sinfty\times X^k_\sinfty=Y^k_\sinfty$ for all $k\in \N$.
 If $\varphi: I_i^k \rightarrow I_\sinfty^k$ is an isomorphism, and $\phi_i$ is $\eta_i$-GH-isometry between $X_i^k$ and $X_\sinfty^k$ with $\eta_i\downarrow 0$ for $i\rightarrow 0$, then for sequence $(\delta_i)_{i\in \N}$ with $\delta_i\downarrow 0$ $\delta_i$-GH-isometries $\Phi_i$ are given by $\Phi_i(s,x)= (\varphi_i(s), \phi_i(x))$. Hence $I_i\times X_i$ converges  to $I_\sinfty\times X_\sinfty$ in the covered GH sense.

We set $f_i^k= f_i|_{I_i^k}$ for $i\in \N\cup\{\infty\}$ and $k\in \N$. As before we can assume w.l.o.g. that $I^k_i=I^k_\sinfty=I^k$ and $\varphi_i(s)=s$. 
It is easy to see, that the family $\{Y_i\}_{i\in \N}$ is uniformly non-totally imprisoning.

We fix $\epsilon>0$ and $k\in \N$. Since $Y^k_\sinfty$ is compact, there exists $L>0$ such that $\diam Y^k_\sinfty\leq L/2$ and for $i$ sufficiently large we have $\diam Y^k_i\leq L$. 

Given $\el \in \N, \epsilon\in (0, \frac 1 \el), f^k_\sinfty, L$ there exist $i(\epsilon, \el, f_\sinfty^k, L)=i_0\in \N$ as well as $\delta(\epsilon, \el, f_\sinfty^k, L)=\delta_0>0$ as in the previous Lemma. 

We continue to show second property (2) in the definition of uniform convergence. 
The first property (1) works similarly. 

We choose  $(s,x), (t,y)\in \{\ell\geq \frac 1 {\el}\}\cap Y^k_\sinfty$.  
Let $(s_i, x_i), (t_i, y_i)\in Y^k_i$ be such that 
$$\de_{I\times X^k_\sinfty}(\Phi_i(s_i, x_i), s,x)+ \de_{I\times X^k_\sinfty}(\Phi_i(t_i, y_i), t, y)\leq \delta \mbox{ for }\delta \in (0, \delta_0). $$
Since $\de_{I\times X^k_i}$ is the product metric,   it follows that 
$$\de_{I\times  [0,L]}(s_i, 0, s, \de_{X_\infty}(\phi_i(x_i),x))+\de_{I\times [0,L]}(t_i, 0, t,\de_{X_\sinfty}(\phi_i(y_i), y))\leq \delta.$$
as well as, by fiber independence, 
\begin{align*}\ell_{-I\times_{f_\sinfty} X_\sinfty}(s,x, t, y)&=\ell_{-I\times_{f_\sinfty}[0,L]}(s, 0, t, \de_{X_\sinfty}(x,y)), \ \ \\
\ell_{-I\times_{f_i} X_i}(s_i, x_i, t_i, y_i)&= \ell_{-I\times_{f_i}[0,L]}(s_i, 0, t_i, \de_{X_i}(x_i, y_i).
\end{align*}
Hence,  the previous lemma implies
$$\ell_{-I\times_{f_i} X_i}(s_i, x_i, t_i, y_i)
\geq \ell_{-I\times_{f_\sinfty} X_\sinfty}(s,x, t, y)- \epsilon. $$
This is the  property (2) of uniform convergence.
\end{proof}
The coverings $\{Y^k_i\}_{k\in \N}$, $i\in \N$, are weak causal convex by definition. Indeed, if $\gamma=(\alpha, \beta)$ is a  maximal geodesic with endpoints in $Y^k_i$, then the enpoints of $\beta$, that is a minimal geodesic in $X_i$ by Theorem \ref{th:fiber_ind},  are in $X^k_i=\bar B_{2^k}(o_i)$. Hence, $\Im \beta\subset \bar B_{2^{k+1}}(o_i)$. Consequently, $\Im\gamma \subset Y_i^{k+1}$. Therefore, we obtain the following corollary. 
\begin{corollary}
Let $Y_i:= (-I_i \times_{f_i} X_i, (Y_i^k)_{k\in \N})$, $i \in \N\cup \{\infty\}$,  be  properly covered generalized cones as in the previous Theorem. Assume $Y_i$ satisfies the condition $\TCD^e_p(K,N)$ for $i \in \N$. Then $Y_\sinfty$ satisfies the condition $w\TCD^e_p(K,N)$. 
\end{corollary}

The proof of Proposition \ref{prop:uniform_convergence_cone} can be generalized to obtain the following theorem.
\begin{theorem}\label{th:uniform}
Let $(M, g_i)_{i\in \N\cup\{\infty\}}$ be a family of globally hyperbolic spaces times with smooth Lorentzian metrics $g_i, i\in \N\cup\{\infty\}$. Let $h$ be a Riemannian metric and assume that $g_i$converges uniformly to $g_\sinfty$ w.r.t. $h$, i.e. $\forall \epsilon>0$ there exists $i(\epsilon)\in \N$ such that 
\begin{align}\label{ineq:lm} g_\sinfty - \epsilon h \leq g_i \leq g_\sinfty + \epsilon h\ \  \ \forall i\geq i(\epsilon)\end{align}
and $g_\sinfty\pm \epsilon h$ are Lorentzian metrics. 

Then, the associated sequence of Lorentzian pre-length spaces $(M, \de^h, \ell^{g_i})$ $\ell$-converges to $(M, \de^h, \ell^{g_\sinfty})$. 
\end{theorem}
\begin{proof}{\bf (0)} 
We pick a sequence $(\epsilon_i)_{i\in \N}$ s.t.  $\epsilon_i\downarrow 0$. For $\epsilon_i>0$ sufficiently small,  $g_\sinfty\pm \epsilon_i h$ is a globally hyperbolic Lorentzian spacetime metric, and we will assume this in the following. 

According to \cite{ms_gh} for every $i\in \N\cup\{\infty\}$ there exists a proper covering $\{U_i^k\}_{k\in \N}$ such that $U_i^k$ is causal convex w.r.t. $g_\sinfty + \epsilon_i h$. 
We choose $\{U_i^k\}_{k\in \N}$ as a proper cover for the spacetime $(M, g_i)$. 
$U_i^k$ is causal convex w.r.t. $g_i$. Indeed, if $\gamma$ is a causal curve w.r.t. $g_i$ with endpoints in $U_i^k$, then $\gamma$ is also a causal curve w.r.t. $g_\sinfty + \epsilon_i h$ and therefore $\mbox{Im}\gamma\subset U^k_i$.

The underlying metric space is $(M, \de^h)$ for every $i\in \N$ in the sequence. Since the convergence of $g_\sinfty +\epsilon_i h$ to $g_\sinfty$ is smooth, from the construction of $U^k_i$ in \cite{ms_gh} it is clear that $\ol U^k_i$ converges in Hausdorff sense to $\ol U^k_\sinfty$ in $(M, \de_h)$ for $i\rightarrow \infty$ for every $k\in \N$. Hence, $\ol U_i^k \subset B_{\delta}^h(\ol U_\sinfty^k)$ for every $i\geq i(\delta, k)$. 

Consequently, $(M, \de^h, \ell_i, \{U^k_i\})$ is uniformly non-totally imprisoning. Indeed, if $\gamma$ is a causal curve w.r.t. $g_i$ in $\ol U_i^k$. Then, $\gamma$ is  a causal curve w.r.t. $g_\sinfty+ \epsilon_0 h$ in ${B_\delta^h(\ol U^k_\sinfty)}$.

Since $g_\sinfty+ \epsilon_0h$ is globally hyperbolic, it is also non-totally imprisoning, we have that $L^h(\gamma)\leq C(k)$ for  constants $C(k), k\in \N$.

{\bf (1)} 
From \eqref{ineq:lm} it follows
$$\ell_{g_\sinfty-\epsilon h} \leq \ell_{g_i} \leq \ell_{g_\sinfty +\epsilon h}.$$

We pick $x,y\in \{ \ell_{g_\sinfty}\geq \frac 1 l\}\cap \ol U^k_i$. Since $g_\sinfty$ is smooth, the time separation function $\tau_{g_\sinfty}$ is continuous. Hence, there exists $\delta(\epsilon, l)$ such that for $\delta\in (0, \delta(\epsilon, l))$ we have 
$$\ell_{g_\sinfty}(x,y) - \ell_{g_\sinfty}(\tilde x,\tilde y)<\epsilon, \mbox{ whenever } \de_h(x,\tilde x)+ \de_h(y, \tilde y)<\delta.$$
We choose such a pair $(\tilde x, \tilde y)$. Hence, there exist future oriented timelike curve $\gamma: [a,b]\rightarrow M$ that connects $\tilde x$ and $\tilde y$ and such that $L^{\tau_{g_\sinfty}}(\gamma)\geq \frac 1 {2l}$.

We parametrize $\gamma$ according to $\tau$-arclength. In particular, $g_\sinfty(\gamma', \gamma')\geq \frac 1 {4l^2}.$

For $\epsilon>0$ sufficiently small, we have \begin{align}\label{another_ineq} g_i(\gamma', \gamma')\geq g_{\sinfty}(\gamma', \gamma') -\epsilon h(\gamma', \gamma')>0. \end{align}  Hence, $\gamma$ is causal w.r.t. $g_\sinfty - \epsilon h$ and w.r.t. $g_i$ and it follows that 
$$\sqrt{g_\sinfty(\gamma',\gamma')- \epsilon h(\gamma', \gamma')}\geq \sqrt{g_\sinfty(\gamma', \gamma')} - \epsilon L^h(\gamma).$$
Since $g_i, i\in \N,$ is uniformly non-totally imprisoning, and because of \eqref{another_ineq} it follows that 
$$\ell^{g_i}(x,y)\geq \ell^{g_\sinfty}(x,y) - \epsilon C(k).$$ 
This is the second property for uniform convergence. 

{\bf (2)}
The first property of uniform convergence can be shown by  contradiction argument as in {\bf (2)} of  the proof of   Proposition \ref{prop:uniform_convergence_cone}.

{\color{black}
We sketch some of the details. We assume there exists $\epsilon'>0$ such that $\forall i\in \N$, there is $(\delta_i)_{i\in \N}$ with $\delta_i\downarrow 0$ and points $x_i, y_i, \tilde x_i , \tilde y_i\in \ol U^k_i$ such that $\ell_i(x_i, y_i)\geq 0$, $\de_h(x_i, \tilde x_i) + \de_h(y_i, \tilde y_i)\leq \delta_i$ but
$\ell_i(x_i, y_i) >\ell_\sinfty(\tilde x_i, \tilde y_i)+\epsilon'.$

There exists a maximal causal curve $\gamma_i$ in $\ol U^{k+1}_i$ w.r.t. $g_i$ between $x_i$ and $y_i$. In particular, $\gamma_i$ is also a causal curve w.r.t. $g_\sinfty + \epsilon_i h$. Since the sequence of spaces  is uniformly non-totally imprisoning, we can extract a subsequence such that $\gamma_i\rightarrow \gamma$ uniformly for a Lipschitz curve $\gamma$ in $\ol U^{k+1}_\sinfty$. We can apply a limit curve theorem (as for instance in \cite{minlim}) to find that also  $\gamma$ is  a causal curve w.r.t. $g_\sinfty+ \epsilon_i h$. 

Now, it is straightforward to finish the proof as in the proof of Proposition \ref{prop:uniform_convergence_cone}, using upper semicontinuity of $L^{g_\infty+ \epsilon_i h}$ on causal curves and lower semi-continuity of $\ell_{g_\sinfty}$. }
\end{proof}
{\color{black} \begin{corollary}
Timelike sectional lower curvature bounds and timelike curvature-dimension bounds are preserved under uniform convergence of  Lorentzian spacetimes. 
\end{corollary}}
\section{Generalized cones and curvature bounds}\label{sec:curcon}
\subsection{Timelike sectional curvature bounded from below}
We recall first the following Theorem by Graf, Alexander, Kunzinger and Saemann. 
\begin{theorem}\label{th:agks}
Let $(X, \de)$ be a geodesic metric space, $I\subset \R$ an open interval, and $f\in C^2(I)$. If 
\begin{enumerate}
\item $f'' + Kf \leq 0$, 
\item $X$ is an Alexandrov space with curvature bounded from below by
$$\underline K= -\inf\{ K f^2 + (f')^2\}.$$
\end{enumerate}
Then $-I\times_f X$ has timelike curvature bounded from below by $-K$. 
\end{theorem}
\begin{remark}\label{rem:contra} 
Unlike as for the metric space analog of this theorem \cite{albi}, here
the reverse direction is in general not true. This follows from (2) in Remark \ref{rem:RK}, the non-monotonicity of the conditon $R\geq K$, together with Proposition \ref{prop:albi} below. More explicitly, one can consider $X=\mathbb R^2$, a space of sectional curvature equal to $0$,  and $f: I\rightarrow (0, \infty)$, $f(t)=1$, that is concave. 
Then $-I\times_f X= -I\times X$ has sectional timelike curvature bounded from below by $0$. Hence, it also has sectional timelike curvature bounded from below by $-K$ for some $K<0$ that can be  arbitrarly negative. At the same time $f$ is also $\mathcal F K$-conave, i.e. $f''+Kf= K\leq 0$.
But $\underline K=-\inf \{ Kf^2 + (f')^2\}= - K$ is positive. 
\end{remark}

Let $I=(a,b), a, b\in \R\cup\{\pm \infty\}$ be an open interval and let $f:   I\rightarrow \R$ be  a function such that $f$ is continuous in $I$ and $f|_I$ is semi-concave, i.e. $\forall t\in I$ there exists $U\subset I$ and $\lambda\in \R$ such that $t\in U$ and $f|_{U}$ is $\lambda$-concave. 

In \cite{petsem} the local slope of $f$ in $x\in I$ is defined by
$$
|\partial  f|(x)= \limsup_{y\rightarrow x} \frac{ (f(y)- f(x))^+}{|x,y|}= \max\left\{ \frac{d^+f}{dt}(x), -\frac{d^- f}{dt}(x), 0\right\}$$
where $\frac{d^+f}{dt}$ and  $\frac{d^-f}{dt}$ are the right and the left derivative of $f$.

 We notice that with this convention  $|\partial f|(x_0)=0$ if $x_0$ is a local maximum point of $f$.
If $f\in C^1(I)$, then $|\partial f|(x)= |f'(x)|$ for all $x\in I$.

\begin{lemma}[{\cite[Lemma 1.3.4]{petsem}}] Let $\ol I_i=[a_i, b_i]\subset [c_i, d_i]=\ol J_i$ be compact intervals such  that $\ol I_i$ and $\ol J_i$ converge in GH sense to compact intervals $\ol I$ and $\ol J$ respectively. 
Let $f_i : \ol J_i \rightarrow \R$ be  a sequence of continuous functions that is $\lambda$-concave on $J$ and  converges uniformily to $f: \ol J\rightarrow \R$. 
 Then
$$\liminf |\partial f_i|(x_i) \geq |\partial f|(x) \mbox{ whenever $x\in \ol I$ and } x_i\rightarrow x.$$
\end{lemma}
\begin{remark}
Let $f: J\rightarrow (0, \infty)$ be $fK$-concave and continuous on $ J$. We consider the function $G= Kf^2+ |\partial f|^2$. It follows that $G$ is lower semi continuous in $J$. 
If $\ol I \subset J$ is compact, then there exists $z\in \ol I$ such that $G(z)= \min_{x\in \ol I} G(x). $ For instance, we can consider the $f0$-concave function $f(x)= 1-x$ for $x\geq 0$ and $f(x)=1+x$ for $x<0$. Then $\min_{x\in (-1,1)} G(x) = \min_{x\in [-1/2, 1/2]} G(x) = |\partial f|(0)=0$. 
\end{remark}

\begin{corollary} 
Let $I_i$ and $J_i$ be as in the previous lemma. 
Let $f_i :\overline J_i \rightarrow \R$ be the sequence of $\mathcal FK$-concave functions that  converges uniformily to $f: \overline J \rightarrow \R$. We assume the properties from the previous lemma. Then 
$$\liminf_{i\rightarrow \infty} \inf_{\ol I} G_i\geq \inf_{\ol I} G. $$
\end{corollary}
\begin{proof}
Let $f: \ol J \rightarrow [0, \infty]$ be $fK$-concave, and let $f_i: \ol J_i \rightarrow [0, \infty)$ be a sequence of $\lambda$-concave functions that converges to $f$. There exists $x_i\in \ol I_i$ such that $\inf_{\ol I} G_i = \min_{\ol I} G_i=G_i(x_i)$. There exists a subsequence $(x_{i_j})$ that converges to $x\in \ol I$. Hence 
$$\lim_{j\rightarrow \sinfty} G_{i_j}(x_{i_j})= \lim_{j\rightarrow \infty} Kf_{i_j}(x_{i_j}) + |\partial f_{i_j}|^2(x_{i_j})\geq G(x)\geq \inf_{\ol I} G.$$
Hence
$$\liminf_{i\rightarrow \infty} \inf_{\ol I} G_i\geq \inf_{\ol I}G .$$

\end{proof}

\begin{1}\label{th:11}
Let $I\subset \R$ be an open interval, let $f:  I\rightarrow (0, \infty)$ be continuous, and 
assume $f$ is $fK$-concave on $I$. Consider $$\sup_I \{ - K f^2 - |\partial f|^2\}= - \inf_I \{Kf^2 + |\partial f|^2\}=K_f.$$ 
Assume $X$ has Alexandrov curvature bounded from below by $K_f$. Then $-I\times_f X$ has curvature bounded from below $-K$. 
\end{1}
\begin{proof}
We pick a cover $\{I^k\}_{k\in \N}$ for $I$ as in Definition \ref{def:covercone}. 

 Given a point $o\in X$, then $U^k= I^k\times B_{2^k}(o)$  is a weak causal convex and proper cover of $-I\times_f X$.

For $\epsilon>0$ the rescaled metric space $(1-\epsilon)^{\frac{1}{2}} X=: X_\epsilon$ has curvature bounded from below by $$(1-\epsilon)K_f=- (1-\epsilon) \inf G\ .$$
Since $\inf G\leq \inf_{\ol I^k} G$,  it holds 
$$(1-\epsilon)K_f \geq - (1-\epsilon)\inf G\geq - (1-\epsilon) \inf_{\ol I^k} G.$$

There exists a sequence $f_i:I \rightarrow [0, \infty)$  of  smooth $\mathcal F K(1-\eta_i)$-concave functions that converges uniformly to $f$ subject to $\{I^k\}$ where $\eta_i\downarrow 0$. Such a sequence can  easily be contructed by molification (compare with \cite{cavmil, cks, ketterer5}). Then $U_i^k:= U^k$ is also a weak convex and proper cover for $-I\times_{f_i} X_\epsilon$.  We set $G_i:= K (f_i)^2 + (f_i')^2$.

For $k\in \N$ fixed, by the previous corollary, there exists  $i\geq i(\epsilon)$ such that 
$$\inf_{\ol I^k} G_i\geq(1-\epsilon) \inf_{\ol I^k} G  \ \mbox{ or } \  -\inf_{\ol I^k} G_i\leq - (1-\epsilon)\inf_{\ol I^k} G  .$$

Hence,
$(1-\epsilon)K_f\geq - \inf_{\ol I^k} G_i.$
Moreover,  $\inf_{\ol I^k} G= \inf_{I^k} G$. 

We can apply Theorem \ref{th:agks} and deduce that $ - I^k \times_{f_i} X_\epsilon$ has timelike sectional curvature bounded from below by $-K(1-\eta_i)$ for $i\geq i(\epsilon).$

In particular, $I^{k-1}\times B_{2^{k-1}}(o)$ is $\geq (-K(1-\eta_i))$-comparison neighborhood in $I\times_{f_i} X_\epsilon$. 

For each $n\in \N$ we pick $i_n\geq i(\frac 1 n)$. Then, it follows from Theorem \ref{th:convergencecones} that $-I\times_{f_{i_n}} X_{1/n}$ $\ell$-converges to $-I\times_f X$. 

It follows from Theorem \ref{th:stabilityTCBB}  (and the remark thereafter) that $I^{k-1}\times B_{2^{k-1}}(o)$ is a $\geq (-K)$-comparsion neighborhood in $-I\times_f X$. 

Since $k\in \N$ was arbitrary, $-I\times_f X$ has timelike curvature bounded from below by $-K$. 
\end{proof}
{\color{black} 
\begin{example}
Let $f: (0,2) \rightarrow [0, \infty)$ given by $f(x)=x$ for $x\leq 1$ and $f(x)=2-x$ for $x>1$. $f$ is concave and $\inf_{x\in (0, 2)} \{ |\partial f |\}= |\partial f|(1)=0$. Let $X$ be an Alexandrov space with nonnegative curvature. Then $-(0,2)\times_f X$ is a Lorentzian geodesic  spaces satisfying $TCBB(0)$. 
\end{example}}
\subsection{Measure contraction property}
\begin{3}
Let $f: I\rightarrow [0, \infty)$ be as before. 
Assume $X$ is a metric measure space that satisfies the timelike curvature-dimension condition $\TCD_p((N-1)K_f,N)$. Then $-I\times^N_f  X$ satifies $\TMCP(NK,N+1)$.
\end{3}
\begin{proof} 
The proof is the same as for Theorem \ref{th:11} where Theorem \ref{th:stabilityTCBB} is replaced with Theorem \ref{th:stabilityTMCP},  Theorem \ref{th:agks} is replace with the corresponding result about generalized cones and the condition $\TMCP$ in \cite{cks} and $\ell$-convergence is replace with measured $\ell$-convergence.
\end{proof}
\subsection{Ricci limit spaces}
\begin{2}
Let $f: I\rightarrow [0, \infty)$ be as before. 
Assume $X$ is a metric measure space that is the limit of Riemannian manifolds $M^n$ equipped with the Riemannian volume form and with $\ric^{M}\geq (n-1)K_f$ for $K_f$ as before. Then $-I\times^n_f X$ satifies the condition $\TCD^e_p(nK, n+1)$.
\end{2}
\begin{proof}
We again approximate $f$ by smooth, $\mathcal FK(1-\eta_i)$-concave functions $f_i:I\rightarrow [0, \infty)$ as in the previous proofs. . 

It follows that $-I\times_{f_i} M^n_i$ satisfies $\ric^{-I\times_{f_i} M^n_i}\geq K $ where $M^n_i=M^n_{\epsilon_i}$ for a sequence $\epsilon_i\downarrow 0$. It follows that $-I\times_{f_i} M^n_i$ satisfies $\TCD^e_p(nK, n+1)$. 

It follows from Theorem \ref{th:stabilityTCD} and from Theorem \ref{th:convergencecones}  that  $-I\times_f X$ satisfies  $w\TCD^e_p(nK,n+1)$.  

Since $X$ is a GH limit of a sequence of $n$-dimensional Riemannian manifolds with Ricci curvature bounded from below, it is non-branching \cite{coldingnaberI}. Hence,  by Theorem \ref{th:fiber_ind} $-I\times_f X$ is timelike nonbraching as well and therefore satisfies  $\TCD_p(nK, n+1)$ \cite{braun_renyi}.
\end{proof}
\section{$\ell$-compactness in the class of generalized cones}\label{sec:preco}
Given a continuous function $f: I\rightarrow (0, \infty)$ and a metric space $X$, $-I\times_f X$ denotes the generalized cone.  In the following we always choose a proper cover for $-I\times_f X$ in the same way as before, and by abuse of notation $-I\times_f X$ will denote the associate properly covered generalized cone.

We consider the following class of  Lorentzian length spaces. 
$$\mathcal Y= \{Y: \exists- I\times_f X \mbox{ s.t. }Y\simeq -I\times_f X\}.$$ 
Let $\mathcal M$ be the class of smooth, complete Riemannian manifolds. We define
$$\mathcal Y^{\ninfty}= \{ Y\in \mathcal Y: Y\simeq -I \times_f M \mbox{ s.t. }f\in C^2(I),   X\in \mathcal M\}$$
and 
$$\mathcal Y^\ninfty_{K,N}= \{ Y\in \mathcal Y^\infty: \ric^Y\geq K, \dim Y= N\}.$$

We  consider every $Y\in \mathcal Y^\sinfty$ equipped with the measure $f^n \de t \otimes \de\vol_M$ where $n=\dim M$.  In particular, if $Y\in \mathcal Y^\sinfty_{K,N}$, then $n+1= N$. We consider $Y\in Y_{K,N}^\sinfty$ as $n$-cone $Y=-I\times_f^n X$, i.e. $-I\times_f X$ equipped with the measure $f^n \de t\otimes \vol_X$ where $\vol_X$ is the Riemannian volume measure. 
\begin{4}
The class $\mathcal Y^{\ninfty}_{K,N}$ is precompact w.r.t. measured $\ell$-convergence.

Moreover, any limit space $Y\in \overline{\mathcal Y^{\ninfty}_{K,N}}^{\m\ell}$ satisfies the   condition $\TCD(K,N)$. 
\end{4}
We recall the following about Ricci curvature for semi-Riemannian warped products.  The formulas of the  Ricci tensor of $- I\times_{f} M$ have been derived in \cite[Proposition 7.43]{oneillsemi}: 
Given $V,W\in TM$, then 
\begin{enumerate}
\item $\Ric_{ I\times_{f}M}(\frac{\partial}{\partial t}, \frac{\partial}{\partial t})= - n \frac{f''}{f} .$
\smallskip
\item $\Ric_{ I \times_{f} M}(\frac{\partial}{\partial t},V)= 0.$
\smallskip
\item $\smash{\Ric_{ I \times_{f} M}(V,W)= \Ric_M(V,W) - \left( \frac{-f''}{f} + (n-1) \frac{-(f')^2}{f^2} \right) \langle V, W\rangle_{I\times_f M}}.$
\end{enumerate}

\begin{proposition}\label{prop:ricciwp}
The warped product  $ I\times_{f} M$ satisfies $\ric_{\mathring I \times_f M}\geq n K$, i.e. 
$$\ric^{\mathring I \times_f M} (v,v)\geq - n K g_{\mathring I \times_f M} (v,v) \ \ \forall v\in TM, $$
if and only if the following two properties hold
\begin{enumerate}
\item $f''+ K f \leq 0$, 
\item $\ric_M\geq (n-1) K_f$ where $K_f= - \inf_{ I} \left\{ Kf^2 + (f')^2\right\}.$
\end{enumerate}
\end{proposition}
\begin{proof} We only show the ``only if'' direction.

If $\ric_{ \mathring I \times_f M} \geq nK $ implies that $M$ has Ricci curvature bounded from in below  in timelike directions. Hence,  for $Z= \frac{\partial}{\partial t}$ we get 
$$-n \frac{f''}{f} = \ric_{\mathring I\times_f M} (\frac{\partial}{\partial t}, \frac{\partial}{\partial t})\geq -n K g_{\mathring I\times_f M}(\frac{\partial}{\partial t}, \frac{\partial}{\partial t})=nK.$$
Consequently $f''+ Kf\leq 0$. 

We choose $ V$ with $g_{ I\times_f M}(V,V)\geq 0$. Then, $\ric_{\mathring I\times_f M}\geq nK$ implies that 
\begin{align*}-Kn g_{\mathring I\times_f M}(V,V)&=-Kn f^2 g_M(V, V)\leq  \Ric_{\mathring I\times_{f}M}(Z, Z)\\
&=   \ric_M(V,V)  -\left( \frac{-f''}{f} + (n-1) \frac{-(f')^2}{f^2}\right)  \langle V, V\rangle\\
&= \ric_M(V,V)  +\left( \frac{f''}{f} + (n-1) \frac{(f')^2}{f^2}\right)  f^2g_M(V,V)
\end{align*}
By rearranging the terms we get
\begin{align*} 
\ric_M(V,V)&\geq- (n-1)\left( K+  \frac{(f')^2}{f^2}\right)  f^2g_M(V,V) - \left( \frac{f''}{f}  +K \right) f^2 g_M(V,V)\\
&\geq -(n-1) \left( Kf^2+  {(f')^2}\right)  g_M(V,V) 
\end{align*}
This yields $\ric_M\geq 
 -(n-1) \inf\{ Kf^2 + (f')^2\}.$
\end{proof}
\begin{remark}
The {\it only if} direction in the previous proposition does not hold in general if the lower bound $\ric^{-\mathring I\times_f M}\geq K$ on the full Ricci tensor is replaced with a timelike lower Ricci curvature bound (Remark after \ref{rem:contra}). 
\end{remark}
\begin{proof}[Proof of Theorem \ref{th:1}] 
Let $Y_i=-I_i\times_{f_i} M^n_i$ be a sequence in $\mathcal Y^\ninfty_{K,N}$.  We rescale the spaces such that  $K$ is replaced with $nK$. 

Let $I^k_i\times B_{2^k}(o_i)= U_i^k$ be a cover for $Y_i$, and choose $\bar t_i\in I_i$ s.t. $\bar t_i\in I^k_i$ for all $k\in \N$. We recall that we assume $I^k\leq \min\{k , \diam I\}$ for all $k\in \N$.

We scale $f_i$ and $M_i^n$ s.t. $f_i(\bar t_i)=1$, i.e. we replace $f_i$ with $f_i(\bar t_i)^{-1} f_i$ and $M^n_i$ with $f_i(\bar t_i) M_i^n$. Clearly, this will not affect the  lower Ricci curvature bound $nK$ of $Y_i$.


It follows from the previous proposition that 
\begin{enumerate}
\item $f_{i}$ is $\mathcal FK$-concave, 
\smallskip
\item  $\ric^{X_i}\geq -(n-1) K_{f_{i}}$ where 
$K_{f_i}=-\inf \{ K f^2_i + (f_i')^2 \}= - \inf G_i.$
\end{enumerate}
Since $\dim Y_i\leq N$, we have $\dim M_i^n=n\leq N-1$.

We extract a subsequence s.t. $I_i^k \overset{GH}{\rightarrow} I^k_\sinfty$ $\forall k\in \N$ for intervals $I^k_\sinfty\subset \R$ such that $I^k_\sinfty\subset I^{k+1}_\sinfty$ $\forall k\in \N$. We set $I_\sinfty= \bigcup_{k\in \N} I^k_\sinfty$
where $I_\sinfty\subset \R$ is again an open interval.  We set $I^k_i=(a^k_i, b^k_i)$ for $i\in \N\cup \{\infty\}$. 

Since $f_i$ is $\mathcal FK$-conave, $\log f_i=u_i$ satisfies the Ricatti equation 
$$u_i'' + (u_i)^2 + K \leq 0.$$
It follows that 
$$|(\log f_i|_{I^k_i})'(t)|\leq \max\left\{ \cot_{K}(t-a^k_i), \cot_{K}(b^k_i-t)\right\}$$
(for instance, compare with \cite[Theorem III.4.3]{chavel}).

It follows that for $k\in \N$ fixed the family $\log f_i|_{I^k_i}=:u^k_i$, $i\in \N$,  is  uniformly Lipschitz  on $I^k_i$. Moreover, we have $u^k_i(\bar t_i)= \log f_i(\bar t_i)= \log 1=0$. 

Hence, for $k\in \N$ fixed, $(u^k_i)_{i\in \N}$ (and therefore $(f_i|_{I^k_i})_{i\in \N}$) is uniformly bounded on $I^k_i$.

Hence, we can extract a subsequence of $(u^k_i)_{i\in \N}$ that converges uniformly to $\tilde u_\sinfty^k: I^k_\sinfty\rightarrow (0, \infty)$. The function $\tilde u_\sinfty^k$ is Lipschitz. 
Clearly, it also follows that $f_i|_{I^k_i}$ converges uniformly to $e^{\tilde u_\sinfty^k}|_{I^k_\sinfty}= f_\sinfty^k|_{I^k_\sinfty}$ and $f_\sinfty^k$ is $\mathcal FK$-concave. 

We find such functions for every $k\in \N$ such that $\tilde u^{k+1}_\sinfty|_{I^k_\sinfty}= \tilde u^k_\sinfty$, and hence a continuous function $\tilde u_\infty$ on $I_\sinfty$. We set $f_\sinfty= e^{\tilde u_\sinfty}$ on $I_\sinfty$ and $f_\sinfty$ is $\mathcal FK$-concave. By a diagonal argument we find a subsequence $i_\alpha$ such that $f_{i_\alpha}$ converges uniformly to $f_\sinfty$ subject to $I^k_i$ converging to $I^k_\sinfty$.  W.l.o.g. we assume the subsequence is the sequence itself.

We fix again $k\in \N$. Since $|(\ln f_i|_{I^k_i})'|$ and $f_i|_{I^k_i}$ are uniformly bounded on $I^k_i$ also $| (f_i|_{I^k_i})'|$ is uniformly on $I^k_i$, $i\in \N$. 
Hence, there exists a constant $C'(k)>0$ such that 
$$ \inf_{I^k_i} G_i  = \inf_{I^k_i} \left\{ K f^2 + |f'|^2\right\}\leq C'(k) \ \forall i\in \N.$$
It follows from Proposition \ref{prop:ricciwp} that $\ric^{X_i}\geq - (n-1)C'(k)$. 
After extracting another subsequence  $X_i$ converges in pointed GH sense to $X_\infty$. 

Consequently, we are in position to apply Theorem \ref{th:convergencecones}. It follows that $-I\times_{f_i} X_i$ $\ell$-converges to $-I\times_{f_\sinfty} X_\sinfty$.  By Theorem \ref{th:stabilityTCD} $-I\times_{f_\sinfty} X_{\sinfty}$ satisfies the $w\TCD(nK,N)$. 

Moreover,  $X_\sinfty$ is the limit of Riemannian manifolds with Ricci curvature bounded from below. Hence, $X_\sinfty$ is nonbranching \cite{coldingnaberI}. It follows that also $-I\times_{f_\sinfty} X_\sinfty$ is timelike nonbranching by Theorem \ref{th:fiber_ind}.  Therefore, $-I\times_{f_\sinfty} X_{\sinfty}$ satisfies the condition $\TCD(nK,N)$ by \cite{braun_renyi}. 
\end{proof}
\begin{remark} 
We actually showed that 
$$\overline{\mathcal Y^\sinfty_{nK,N}}^{\m\el}\subset \left\{ -I\times_f X:   f: I\rightarrow (0,\infty),  \mathcal{F}K\mbox{-concave}, X \in\overline{ \mathcal X^\sinfty_{(n-1)K,n}}^{pGH}\right\}.$$

\end{remark}
\begin{example}
Let $f_i: (0, \pi) \rightarrow [0, \infty)$, $f(t)= (\cos t)^{\frac 1 i}$, and let $X\in \mathcal M$ with $\ric^X\geq 0$.  $f_i$ is concave and $\inf_I (f')^2 =0$. It follows that $-(0, \pi)\times_{f_i} X= Y$ satisfies $\ric^Y\geq 0$.  The sequence $f_i$ converges locally uniformly in $I$ to $f_\sinfty\equiv 1$. Hence, $Y_i$ converges in measured $\ell$-sense to $I\times X$.
\end{example}
\begin{remark}\label{rem:thelast}
The lower bound on the full Ricci tensor is  necessary in the following sense. A timelike lower Ricci curvature bound will not imply existence of a GH converging subsequence of the fibers. Indeed, we can consider as sequence of Riemannian manifolds $M_i$ with $\ric^{M_i}\geq -(n-1)\lambda_i$ for $\lambda_i\rightarrow \infty$, such that $M_i$ does not have a GH-converging subsequence. We can choose $f_i=\sqrt{\lambda_i} \sin_{K_i}: [0, \pi_{K_i}] \rightarrow [0, \infty)$ and  $\lambda_i$ such that $K_if_i^2 + (f'_i)^2= \lambda_i$. We define $(0, \pi_{K_i})\times_{f_i} M_i$. By construction this space has timelike Ricci curvature bounded from below by $K_i>0$. Hence, by monotonicity it also has timelike Ricci curvature bounded from below $0$, dimension bounded from above by $N$, and the diameter of the interval is bounded. But, by construction, there is no subsequence such that the fiber spaces converge in Gromov-Hausdorff sense. 
\end{remark}
We also define
$\mathcal Y^\dagger_{\kappa, n}= \{ Y\in \mathcal Y^\ninfty: Y \mbox{ satisfies $R^Y\geq -\kappa$ and $\dim_Y\leq n$}\}.$

\begin{5}\label{th:copt}
The class $\mathcal Y^\dagger_{\kappa, n}$ is precompact w.r.t.  covered $\ell$-convergence, and any limit space $Y\in \overline{\mathcal Y^\dagger_{\kappa, n}}^\ell$ satisfies $TCBB(\kappa)$. 
\end{5}
For the proof we can  argue exactly as before after recalling the following proposition in \cite{albi_lorentz}.
\begin{proposition}[{\cite[Proposition 7.1]{albi_lorentz}}]\label{prop:albi}
The warped product  $\mathring I\times_{f} M$ satisfies $R^{-\mathring I\times_f M}\geq -K$
if and only if the following two properties hold
\begin{enumerate}
\item $f''+ K f \leq 0$, 
\item $R^M\geq K_f$ where $K_f= - \inf_{ I} \left\{ Kf^2 + (f')^2\right\}.$
\end{enumerate}
\end{proposition}
\subsection{Tangent spaces}\label{subsec:tangentcone}
Given $Y\in \overline{ \mathcal Y^\ninfty_{NK,N+1}}^{\m \ell}$ we find $f: I\rightarrow [0, \infty)$ $\mathcal FK$-concave and $X$ that is the pointed GH limit of $X_i\in \mathcal M$ with $\ric^{X_i}\geq -(N-1)C$ for a constant $C>0$ such that $Y=-I \times_f X$.
We can rescale $Y$ with $\frac{1}{\epsilon}$ as $\frac 1 \epsilon Y= - I_{\frac 1 \epsilon} \times_{f_{\epsilon}} X_{\frac 1 \epsilon} $ where $I_{\frac 1 \epsilon} = \frac 1 \epsilon I$, $f_{\epsilon}= f(\epsilon \cdot)$ and $X_{ \frac 1 \epsilon} = \frac 1 \epsilon X$ and $\frac 1 \epsilon Y$ is still in $ \overline{ \mathcal Y^\ninfty_{NK,N+1}}^{\m \ell}$.
\begin{definition}
We call  a pointed space $(Y_\sinfty, o_\sinfty) \in \overline{ \mathcal Y^\ninfty_{NK,N+1}}^{\m \ell}$ a tangent cone of $Y$ at some point $(s,x)\in Y$ if there exists a sequence $\epsilon_i$ 
such that $\epsilon_i\downarrow 0$ and $(\frac 1 {\epsilon_i} Y, (s, x))\overset{pt\m\ell}{\rightarrow} (Y_\sinfty, o_\sinfty).$
\end{definition}
\begin{corollary}
Every tangent cone $(Y_\sinfty, o_\sinfty)$ at  $(s,x)\in Y=-I\times_f X$ is isomorphic to $(-\R\times X_\sinfty, (0, o))$ where $(X_\sinfty, o)$ is a tangent cone of $X$ at $x$. 
\end{corollary}
 \begin{proof}
 We observe that $(X,x)$ is the pointed GH limit  of a sequence $(X_j, x_j)\in \mathcal M$ with $\ric^{X_j}\geq -(N-1) C$. Hence, after extracting a subsequence, $(\frac 1 {\epsilon_j} X, x)$ converges in  the pointed GH sense to $(X_\sinfty, x_\sinfty)$, and $X_\sinfty$ is a tangent of $X$ at $x$.  Moreover, $f_{\epsilon_i}: \frac 1 {\epsilon_i} I\rightarrow (0, \infty)$ converges locally uniformly to $f_\sinfty: \R\rightarrow (0, \infty)$ that is concave. We normalize $f_{\epsilon_i}$ and $\frac 1 {\epsilon_i} X$ as in the beginning of the proof of Theorem \ref{th:5}, i.e. such that $f_{\epsilon_i}(s)=1$. It follows that $f_\sinfty(s)=1$ and consequently $f_\sinfty\equiv 1$.  We can apply Theorem \ref{th:convergencecones}. Therefore $\frac 1 {\epsilon_i} Y$ converges to $-\R\times X_\sinfty$  in pointed $\ell$-sense. 
 \end{proof}

\subsection{Almost rigidity in the class of generalized cones} \label{subsec:almost}
\begin{corollary}
We assume that $\epsilon_i\downarrow 0$ and $L_i\uparrow \infty$. If  $Y_i\in \mathcal Y^\ninfty_{-\epsilon_i^2(N-1), N}$ and there exists a timelike maximal geodesic of $\tau$-length bigger than $L_i$ in $Y_i$
then  $Y_i$ subconverges in measured $\ell$-sense to $Y_\sinfty$   such that $Y_\sinfty= -\R\times X$ for some $\RCD(0,N-1)$ space $X$.
\end{corollary}
\begin{proof}
Indeed, we find that, after extracting a subsequence $Y_i$ converges in measured $\ell$-sense to $Y_\sinfty\in \overline{ \mathcal Y^\ninfty_{0,N}}^{\m \ell} $, i.e.  $Y_\sinfty$ satisfies the condition $\TCD(0,N)$ and such that there exists a maximal $\tau$-geodesic $\gamma=(\alpha,\beta)$ of infinite $\tau$-length. Since $f$ is concave along $\gamma$, $f$ must be constant along $\gamma$. Hence, $f\circ \alpha\equiv c$ is constant. Therefore, $Y_\sinfty=\R\times c X$ for a  $\RCD(0, N)$ space $X$. 
\end{proof}
\noindent
{\bf Data availability:} Data sharing not applicable to this article as no datasets
were generated or analysed during the current study.
\medskip

\noindent
{\bf Conﬂict of interest:} The author has no conﬂict of interest.

\bibliography{new} 
\bibliographystyle{plain}
\end{document}